\documentclass[11pt]{amsart}
\usepackage{amsmath,amsthm,amsfonts,amssymb}
\usepackage{enumerate}
\usepackage{hyperref}
\usepackage[T1]{fontenc}
\usepackage{bbm}
\usepackage{color}

\newtheorem{thm}{Theorem}[section]
\newtheorem{lem}[thm]{Lemma}

\newtheorem{cor}[thm]{Corollary}
\newtheorem{pro}[thm]{Proposition}

\theoremstyle{definition}

\theoremstyle{Remark}

\newcommand{\Ind}[1]{\mathbbm{1}_{#1}}
\newcommand{\id}{e_{\BA}}
\newcommand{\ta}{a_v}
\newcommand{\tc}{c_{v}}
\newcommand{\C}{\mathbb{C}}

\newcommand{\mF}{\mathcal{F}}

\newcommand{\R}{\mathbb{R}}

\newcommand{\st}{\cH}

\newcommand{\tm}{\tilde{m}}

\newcommand{\tK}{\tilde{K}}

\newcommand{\mbc}{[\check\bc^{-1}m_B]}

\newcommand{\roq}{\rho_q}
\newcommand{\rop}{\rho_p}

\DeclareMathOperator{\Ima}{Im}

\DeclareMathOperator{\loc}{loc}
\DeclareMathOperator{\glo}{glo}
\DeclareMathOperator{\Lip}{Lip}

\theoremstyle{plain}
\newtheorem{theorem}{Theorem}[section]
\newtheorem*{theorem*}{Theorem}

\theoremstyle{remark}
\newtheorem{remark}[theorem]{Remark}
\numberwithin{equation}{section}
\theoremstyle{definition}

\numberwithin{equation}{section}
\theoremstyle{Basic assumptions}
\newtheorem{Basic assumptions}[theorem]{Basic assumptions}
\theoremstyle{notation}

\renewcommand\Re{\operatorname{\mathrm{Re}}}
\renewcommand\Im{\operatorname{\mathrm{Im}}}

\newcommand\rest[1]{\kern-.1em
          \lower.5ex\hbox{$\scriptstyle #1$}\kern.05em}

\renewcommand\mod[1]{\vert{#1}\vert}
\newcommand\bigmod[1]{\bigl\vert{#1}\bigr|}

\newcommand\bignorm[2]{\,{\big\Vert{#1}\big\Vert_{#2}}}

\newcommand{\ap}{|a|}

\newcommand\wrt{\,\text{\rm d}}

\newcommand\bC{\mathbf{C}}

\newcommand\BA{\mathbb{A}}

\newcommand\BC{\mathbb{C}}

\newcommand\BG{\mathbb{G}}

\newcommand\BK{\mathbb{K}}
\newcommand\BN{\mathbb{N}} \newcommand\OBN{\overline{\mathbb{N}}}
\newcommand\BR{\mathbb{R}} 

\newcommand\BS{\mathbb{S}}

\newcommand\BX{\mathbb{X}}

\newcommand\fra{\mathfrak{a}}

                               \newcommand\frg
{\mathfrak{g}}
\newcommand\cH{\mathcal{H}}

\newcommand\frk{\mathfrak{k}}

\newcommand\cM{\mathcal{M}}

\newcommand\frn{\mathfrak{n}}

\newcommand\frp{\mathfrak{p}}

\newcommand\al{\alpha}
\newcommand\be{\beta}
    \newcommand\Ga{\Gamma}
\newcommand\de{\delta}
  \newcommand\vep{\varepsilon}
\newcommand\ka{{\kappa}}
\newcommand\la{\lambda}   
\newcommand\om{\mathcal E}

\newcommand\vp{\varphi}

\newcommand\OV{\overline}
\newcommand\funnyk{k\hbox to 0pt{\hss\phantom{g}}}

\newcommand\lu[1]{L^1(#1)}
\newcommand\lp[1]{L^p(#1)}

\newcommand\ld[1]{L^2(#1)}

\newcommand\Cvp[1]{Cv_p(#1)}

\newcommand\Mp[1]{\mathcal M_p(#1)}

\newcommand\bc{\mathbf{c}}

\newcommand\wt{\widetilde}
\newcommand\whH{\widehat{\phantom{G}}\hbox to 0pt{\hss $H$}}

\newcommand\emspace{\hbox to 6pt{\hss}}

{\mathsurround=0pt}
\newcommand\rmii{\hbox{\rm (ii)}}

\newcommand\e{\mathrm{e}}

\newcommand\TWp{T_p}
\newcommand\TWq{T_q}
\newcommand\TWo{T_1}

\newcommand\alza[1]{\kern+.13em
          \raise.35ex\hbox{$\scriptstyle #1$}}
\newcommand\alzaprima[1]{\kern-.25em
          \raise.2ex\hbox{$\scriptstyle #1$}}

\DeclareSymbolFont{EUEX}{U}{euex}{m}{n}

\DeclareSymbolFont{euexlargesymbols}{U}{euex}{m}{n}
\DeclareMathSymbol{\intop}{\mathop}{euexlargesymbols}{"52}
     \def\int{\intop\nolimits}

\DeclareSymbolFont{euexsymbols}     {U}{euex}{m}{n}
\title[Improved multiplier theorems on symmetric spaces]{Improved multiplier theorems on rank one noncompact symmetric spaces}

\author{B\l a\.{z}ej Wr\'obel}
\address{Institute of Mathematics, Polish Academy of Sciences, \'Sniadeckich 8,	00–656 Warszawa, Poland \& Institute of Mathematics, University of  Wroc\l aw, pl. Grunwaldzki 2, 50-384 Wroc\l aw, Poland}
\email{blazej.wrobel@math.uni.wroc.pl}



\begin{document}

\begin{abstract}
We prove  multiplier theorems on rank one noncompact symmetric spaces which improve aspects of existing results. A common theme of our main results is that we partially drop specific assumptions on the multiplier function such as a Mikhlin-H\"ormander condition. These are replaced by a requirement that parts of the multiplier function on the critical $p$ strip give rise to bounded Fourier multipliers.
\end{abstract}

\dedicatory{\large{In memory of Jacek Zienkiewicz.}}

%VERSIONE DI \today
\maketitle
\numberwithin{equation}{section}
%\tableofcontents
\section{Introduction}

\numberwithin{equation}{section}

A celebrated result of S.G.~Mikhlin \cite{Mikhlin} and L.~H\"ormander \cite{Ho} states that if
$B$ is a bounded translation invariant operator on $\ld{\BR^n}$
and the Fourier transform $m_B$ of its convolution kernel
satisfies the following Mikhlin-H\"ormander type conditions
\begin{equation*}
\mod{D^I m_B (\xi)}
\leq C \, \mod{\xi}^{-\mod{I}}
\qquad  \xi \in \BR^n \setminus \{0\}
\end{equation*}
for all multi-indices $I$ of length $\leq [\!\![n/2]\!\!]+1$,
then $B$ extends to an operator bounded on $\lp{\BR^n}$ for all
$p\in (1,\infty)$.
The operator $B$ is usually referred to as the Fourier
multiplier operator associated to the multiplier $m_B$.

The problem of extending the classical Mikhlin-H\"ormander multiplier theorem to the setting
of symmetric spaces of the noncompact type has been considered by a number of  authors
\cite{A1,AL,CS,GMM,I1,I2,MV,ST}. We refer the interested reader to these papers an the references therein. Our article improves aspects of the works of Giulini, Mauceri, and Meda \cite{GMM}, and Ionescu \cite{I2}.

Suppose that $\BX= \BG/\BK$ is an $n$-dimensional symmetric space of the noncompact type. Throughout the paper we focus on the case  when $\BX$ has real rank one. 
It is well known that if $B$ is a $\BG$-invariant bounded linear operator on $\ld{\BX}$, 
then there exists a $\BK$--bi-invariant tempered distribution $k_B$ on $\BG$ such that 
$Bf = f*k_B$ for all $f$ in $\ld{\BX}$ (see \cite[Prop.~1.7.1 and Ch.~6.1]{GV} for details).
We call $k_B$ the kernel of $B$.
We denote its spherical Fourier transform~$\wt k_B$
by $m_B$ and call it the spherical multiplier associated to $B$.
Then $m_B$ is a bounded Weyl-invariant function. In the rank one case Weyl-invariance means that $m_B$ is even. 

A well known result of J.L.~Clerc and E.M.~Stein \cite{CS} implies that 
if $B$ is $\BG$-invariant bounded linear operator on $\lp{\BX}$ for some $p$ in $(1,\infty)\setminus\{2\}$, then $m_B$
continues analytically to a bounded Weyl-invariant function in the strip  $$\TWp := \{\zeta \in \BC:  \bigmod{\Im \zeta} < \rop\},$$ where $$\rop:= \de(p)|\rho|\quad\textrm{with}\quad \de(p):=\mod{2/p-1}$$ and $\rho$ is half the sum of all positive roots with multiplicity. The Clerc-Stein result naturally rises the question which conditions on $m_B$ will guarantee the $\lp{\BX}$ boundedness of $B.$

The best sufficient Mikhlin-H\"ormander-type condition available in the literature
is due to A.D.~Ionescu \cite[Theorem 8]{I2}. He proved that if 
$p$ is in $(1,\infty)\setminus\{2\}$, and $m_B$ is a bounded Weyl-invariant holomorphic function in $\TWp$ 
that satisfies the Mikhlin-H\"ormander condition
$$
\bigmod{\partial^jm_B(\zeta)}
\leq C\, \Big(\min \big[\mod{\zeta-i\rop},\mod{\zeta+i\rop}\big]\Big)^{-j},
\qquad  \zeta \in \TWp,
$$
%}
for $j=0,1,\ldots,N$, with $N$ large enough,
then the associated multiplier operator is bounded on $\lp{\BX}.$  Analyzing Ionescu's proof one sees that it is sufficient to take $N>(n+3)/2.$

A different condition that guarantees the $\lp{\BX}$ boundedness of $B$ was established by Giulini, Mauceri, and Meda \cite[Theorem 3.1]{GMM}.  Let 
\begin{equation}
	\label{eq:omdef}
	\omega(\la)=(\la^2+4\rho^2)^{(n-1)/4},\qquad \la\in \TWo,
	\end{equation} 
and note that $\omega$ is a smooth weight of polynomial growth of order $\la^{(n-1)/2}$ at infinity. It was justified in \cite{GMM} that if  $m_B$ is a bounded Weyl-invariant holomorphic function in $\TWp$ for which  $(\omega m_B)(\cdot +i\rho)$ is an $L^p$ Fourier multiplier,  then $m_B$ defines a bounded operator on $\lp{\BX}$.   
  
The purpose of this work is to prove results which enhance aspects of both  \cite{GMM} and \cite{I2}. Our first theorem imposes the assumption (1) on $\omega m_B$ on the correct strip $T_p$ at the prize of (2) - a Mikhlin-H\"ormander condition on a smaller strip $T_q.$  In  \cite[Theorem 3.1]{GMM} the authors did not need (2). On the other hand (1) was assumed on strips larger than $T_p,$ mostly on the largest strip $T_1.$ 

\begin{theorem} \label{t: main} 

Let $p$ is in $(1,\infty)\setminus \{2\}$ and $N>(n+3)/2$. Assume that $m_B$ is a bounded Weyl-invariant holomorphic function in $T_p$ such that
\begin{enumerate}
	\item $(\omega m_B)(\cdot + i \rop)$ is an $\lp{\R}$ Fourier multiplier,
	\item  For some $q\in [p,p']$  which satisfies $|\rho| |1/p-1/q|<1$ we have
$$
\bigmod{\partial^jm_B(\zeta)}
\leq C\, \Big(\min \big[\mod{\zeta-i\roq},\mod{\zeta+i\roq}\big]\Big)^{-j}
\qquad  j=0,\ldots, N,\quad \zeta \in \TWq.
$$
\end{enumerate}
Then $B$ extends to a bounded operator on $\lp{\BX}$.
\end{theorem}

Since $|\rho|<2$ implies  $|\rho| |1/p-1/2|<1$ we immediately obtain a corollary in the case of small $|\rho|.$ Condition $|\rho|<2$ if satisfied in low dimensional hyperbolic spaces $\mathbb{H}^3$ and $\mathbb{H}^4.$ 
\begin{cor} \label{c: main} 
	
Let $|\rho|<2,$   $p\in(1,\infty)\setminus \{2\},$ and $N>(n+3)/2$. Assume that $m_B$ is a bounded Weyl-invariant holomorphic function in $T_p$ such that
	\begin{enumerate}
		\item $(\omega m_B)(\cdot  + i \rop)$ is an $\lp{\R}$ Fourier multiplier,
		\item  $m_B$ satisfies the Mikhlin-H\"ormander condition on $\R$
		$$
		\bigmod{\partial^jm_B(\zeta)}
		\leq C\, |\zeta|^{-j}
		\qquad  j=0,\ldots, N,\quad \zeta \in \R.
		$$
	\end{enumerate}
	then $B$ extends to a bounded operator on $\lp{\BX}$.
\end{cor}

Our second main goal is to improve Ionescu's multiplier theorem \cite[Theorem 8]{I2}. Note that in Theorem \ref{t: main i} below we do not assume any specific condition on $m_B(\zeta)$  locally for $|\Re \zeta|<1.$
\begin{theorem} \label{t: main i} 
	
	Let $p$ is in $(1,\infty)\setminus \{2\}$ and $N>(n+3)/2$. Assume that $m_B$ is a bounded Weyl-invariant holomorphic function in $T_p$ such that
	\begin{enumerate}
		\item $m_B(\cdot + i \rop)$ is an $\lp{\R}$ Fourier multiplier,
		\item  For large values of $\zeta \in \TWp,$ $|\Re \zeta|>1,$ the multiplier $m_B$ satisfies
		$$
		\bigmod{\partial^jm_B(\zeta)}
		\leq C|\zeta|^{-j}.
		\qquad  j=0,\ldots, N.
		$$
	\end{enumerate}
	Then $B$ extends to a bounded operator on $\lp{\BX}$.
\end{theorem}

Our paper may be thought of as an attempt to obtain a counterpart of \cite{CMW} on noncompact symmetric space. There, in collaboration with Celotto and Meda, we considered a similar problem on homogeneous trees.  We obtained a characterization of the class of $L^p$ spherical multipliers on homogeneous trees in terms of $L^p$ multipliers on the torus. Analysis on homogeneous trees often serves as a model for global analysis on rank one noncompact symmetric spaces. Yet it is unclear if a characterization analogous to \cite{CMW} is feasible for spherical multipliers on $\lp{\BX}$. A significant step towards it would be to erase item (2) in Theorem \ref{t: main}, Corollary \ref{c: main} or Theorem \ref{t: main i}. We hope to return to this topic in the future.

Our paper is organized as follows.  Section~\ref{s: Background} contains notation
concerning  rank one symmetric spaces and
the spherical Fourier analysis on these spaces. In Lemma \ref{lem:CtoI} we also give an important inequality for passing between Cartan and Iwasawa coordinates. In Section \ref{s:elg} we split the corresponding convolution operators onto local and global parts and estimate these parts separately. The local parts in both Theorems \ref{t: main} and \ref{t: main i} are estimated by a straightforward  reference to the results from \cite{GMM} and \cite{ST}, respectively. The analysis of global parts of the kernels is more complicated. Here further splittings are needed, together with a transference result from \cite{CMW} (which generalizes
previous results of Ionescu \cite{I2}). We remark that in order to reach the desired results the splittings need to be done in a more refined way than in \cite{I2} or \cite{MW}. This refinement is needed to avoid some $L^1$ estimates. In particular, in the proof of Theorem \ref{t: main} we cannot use the Abel transform argument as in \cite[Lemma 5]{I2} or \cite[Lemma 2.7 iii]{MW}. Finally, in Section \ref{s pf mult} we prove our main results - Theorems \ref{t: main} and \ref{t: main i}. This is achieved by using the results from previous sections together with Harish-Chandra expansions of spherical functions. Since the computations in Section \ref{s pf mult} are  standard we do not provide all the details there.

A duality argument shows that when proving Theorems \ref{t: main} and \ref{t: main i} it is enough to focus on $p\in (1,2).$ Therefore, in the reminder of the paper we always consider $p\in (1,2),$ even if this is not stated explicitly.

We will use the ``variable constant convention'', and denote by $C,$
possibly with sub- or superscripts, a constant that may vary from place to
place and may depend on any factor quantified (implicitly or explicitly)
before its occurrence, but not on factors quantified afterwards. For two non-negative real numbers $X$ and $Y$ by $X\lesssim Y$ we mean that $X\le C Y,$ where $C>0$ is a constant independent of significant quantities. 

\section{Background and preliminary results} \label{s: Background}

\subsection{Preliminaries on symmetric spaces of rank one}
We use notation from \cite[Section 2]{MW}. We recall it here for the sake of completeness. The books of S.~Helgason \cite{H1,H2} and Gangolli and Varadarajan \cite{GV}
are basic references to the subject.

Suppose that $\BG$ is a noncompact semisimple Lie group with finite center.
Denote by $\BK$ a maximal compact subgroup of $\BG$
and consider the associated Riemannian symmetric space of the noncompact type $\BX:=\BG/\BK$.
We briefly summarize the main features of spherical harmonic analysis on $\BX$.

Denote by $\theta$ a Cartan involution of the Lie algebra $\frg$ of $\BG$,
and write $\frg = \frk \oplus \frp$ for the corresponding
Cartan decomposition. Let $\fra$ be a maximal
abelian subspace of $\frp$, and denote by $\fra^*$ its dual space,
and by $\fra_\BC^*$ the complexification of $\fra^*$. Since $\BX$ has rank one the algebra $\fra$ is one-dimensional. We denote by $\BA$ the multiplicative group $\exp\fra$ and let $\id$ be its identity element. Clearly, $\BA$ is then
isomorphic to the additive group of the vector space $\fra$.
It is convenient to choose a particular isomorphism between $\BA$ and $\fra$,
which we now describe.  Denote by $\al$ the unique simple positive root of the pair $(\frg,\fra)$.  
Denote by $H_0$ the unique vector in $\fra$ such that
$\al(H_0) = 1$, and normalize the Killing form of $\frg$ so that $\bigmod{H_0}^2 = 1$.
Every vector in $\fra$ is of the form $tH_0$, with $t$ in~$\BR$, and every element
of $\BA$ is then of the form $\exp(tH_0)$. The Weyl chamber is $\fra^+ := \{tH_0: t>0\}$, and we
set $\BA^+ := \exp (\fra^+)$ and $\BA^- := \exp (-\fra^+)$.  We shall often write $a$ for $\exp(tH_0)$ and $\log a$ for $tH_0$.

The root system is either of the form $\{-\al,\al\}$ or of the form
$\{-2\al,-\al, \al,2 \al\}$.  We denote by~$m_\al$ and $m_{2\al}$ the
multiplicities of $\al$ and $2 \al$, respectively.  Observe that $m_\al+m_{2\al}+1=n$
and $m_{2\al} = 0$ in the case where $2\al$ is not a root.  
Define $\rho$ by  $2\rho = (m_\al+2m_{2\al}) \, \al$, and $2\mod{\rho} = m_\al+2m_{2\al}$.  
We consider the Lie algebras $\frn := \frg_{\al} + \frg_{2\al}$ and $\OV\frn := \frg_{-\al} + \frg_{-2\al}$,
where $\frg_\be$ denotes the root space associated to the root $\be$, and the corresponding connected and 
simply connected nilpotent Lie groups $\BN$ and $\OBN$, respectively.  
If $a$ is in $\BA$ and $\la$ belongs to $\fra^*$, we 
write $a^\la$ instead of $\e^{\la(\log a)}$. In particular, if $a=\exp(tH_0)$ then $a^{\alpha}=e^t.$ 
%Slightly abusing notation, for $a,b\in \BA$ we write
%\begin{equation}
%	\label{eq:abine}
%	a\le b \qquad\textrm{whenever}\qquad a^{\alpha}\le b^{\alpha}.
%\end{equation}  
Define
$$
\de(a)
:= 2^{-2\mod{\rho}} \, \big[a^\al-a^{-\al}\big]^{m_\al}  \,  \big[a^{2\al}-a^{-2\al}\big]^{m_{2\al}}
\qquad  a \in \BA^+
$$
and note that  $\delta(a)$ is of order $(\log a)^{n-1}$ when $a$ is close to $e_{\BA}$ and of order $a^{2\rho}$ when $a$ is large.

The group $\BG$ admits the Cartan decomposition $\BG= \BK\BA^+\BK$ and the Iwasawa decomposition $\BG= \OBN\BA\BK$.  The integration formula in the Iwasawa coordinates is
\begin{equation*} 
	%\label{f integration Iwasawa}
	\int_\BG f(x) \wrt x
	=  c_\BG\, \int_{\OBN} \int_{\BA} \int_{\BK} f(v ak) \, a^{2\rho} \, \wrt v \wrt a\wrt k.
\end{equation*}
We remark that here and later on we always denote by $da$ the multiplicative Haar measure on $\BA.$ This means for instance that
\[
\int_{\R}f(\exp(t H_0))\,dt=\int_{\BA}f(a)\,da.
\]
 
We abbreviate $$\BS:=\OBN \BA.$$ If $k$ is a $\BK$--bi-invariant kernel on $\BG$ then the operator $f\mapsto f*k$ is bounded on $L^p(\BG)$ if and only if $f\mapsto f*_{\BS}K$ is bounded on $L^p(\BS).$ Here $K\colon \BS\to \C$ is given by $K(va):=k(va).$

For each $g$ in $\BG$ we denote by $[g]_+$ and $\exp\big[H(g)\big]$ the middle components of $g$ in the 
Cartan and the Iwasawa decompositions, respectively. 
Recall that in the rank one case $H(v)$ is in $\OV{\fra^+}$ for every $v$ in $\OBN,$ see \cite[Corollary~6.6]{H1}. We set 
\begin{equation}
	\label{e Pv}
	P(v) := \e^{-\rho H(v)},\qquad v\in \OBN.\end{equation}
It is well known that for any $q>1$
\begin{equation}
\label{e Pv'}
(1+ |\alpha (H(v))|)P(v)^q\quad\textrm{belongs to }\lu{\OBN}.
\end{equation}  
This follows 
by an explicit computation starting from \cite[Theorem~6.1~\rmii]{H3}.

An important ingredient in \cite{I2} was the following formula valid for every $a\in \BA^+$ and every $v\in\OBN$
\begin{equation*}
	[va]_+
	= a  \,\exp\big[H(v)\big] \, \exp\big[E(v,a)H_0\big]  
	\qquad\hbox{and}\quad 0\leq E(v,a)\leq 2a^{-2\alpha},
\end{equation*}
see \cite[Lemma 3]{I2}. Abbreviating 
$$b=[va]_+,\qquad \ta=a\exp(H(v))$$ 
the above identity can be restated
\begin{equation}
 \label{f: decom Iwas}
 \ta^{\alpha}\le b^{\alpha} \le \ta^{\alpha}\exp(2a^{-2\alpha}),\qquad a\in \BA_+,\quad v\in \OBN.
\end{equation}
We shall need a related result, which allows general $a\in \BA$. We let $$\Phi(x)=\frac{x+\sqrt{x^2-4}}{2},\qquad x\ge 2.$$
	Then $\Phi$ is the inverse function of $x\mapsto x+1/x$ considered on the domain $[1,\infty).$ It has the following easily verifiable properties
	\begin{equation}
		\label{e Phi}
\Phi(2)=1,\qquad\Phi \textrm{ is increasing,}\qquad 0\le x-\Phi(x)\le \frac{2}{x}.
\end{equation}

\begin{lem}
	\label{lem:CtoI}
	Let $a\in \BA,$ $v\in \OBN$ and denote $b=[va]_+$ and $\ta=a\exp(H(v)).$  Then  we have
	\begin{enumerate}
		\item $b^{\alpha}\geq \Phi\left[\max (a^{-\alpha},(\ta)^{\alpha},2)\right]$
		\item $b^{\alpha}\leq  \Phi\left[a^{-\alpha}+(\ta)^{\alpha}\right]$ 
	\end{enumerate}
Moreover, if $(\ta)^{\alpha}\ge 2$ then it holds 
\begin{enumerate}
	\setcounter{enumi}{2}
		\item$(\ta)^{\alpha} (1-2(\ta)^{-2\alpha})\le b^{\alpha}\le (\ta)^{\alpha} (1+a^{-\alpha}\ta^{-\alpha}),$ \end{enumerate}
	while for $a^{-\alpha}\ge 2$ we have 
	\begin{enumerate}
		\setcounter{enumi}{3}
		\item $a^{-\alpha}(1-2a^{2\alpha})\le b^{\alpha}\le a^{-\alpha} (1+a^{\alpha}(\ta)^{\alpha}).$
	\end{enumerate}
	
\end{lem}
\begin{remark}
	\label{rem:roredu}
	When $2\alpha$ is not a root, then 2) becomes an equality. Whether this helps in the analysis in this particular case is not clear to us at this stage. 
\end{remark}
\begin{remark}
	\label{rem:informal} Items 3) and 4) of the lemma say the following. Either $b^{\alpha}\approx (\ta)^{\alpha}$ or $b^{\alpha} \approx a^{-\alpha}$ according to whether $(\ta)^{\alpha}$ or $a^{-\alpha} $ is large. In applications the comparison between $(\ta)^{\alpha}$ and $a^{-\alpha}$  is going to determine which of the items 3) or 4) will be useful.
\end{remark}
\begin{remark}
	\label{rem:ourvsIon}
For $a\in \BA^+$ item 3) is close to \eqref{f: decom Iwas}. However,  later on we will need to consider also $a\in \BA^{-}$ when estimating global parts of the considered kernels.  Note that for certain $a$ and $v$ we may have $\ta^{\alpha}\ge 2$ even though $a\in \BA^{-}.$   
\end{remark}

\begin{proof}[Proof of Lemma \ref{lem:CtoI}]
	We use explicit computations in \cite[Chapter 2, Theorem 6.1 (i)]{H3}. In our notation we see that
	\begin{equation*}
		\left[\frac{b^{\alpha}+b^{-\alpha}}{2}\right]^2=	\left[\frac{a^{-\alpha}}{2}+\frac{a^{\alpha}}{2}(1+c|X|^2)\right]^2+ca^{2\alpha}|Y|^2.
	\end{equation*}
Here $X$ and $Y$ are coordinates of $v\in \OBN $ corresponding to root spaces $\frg_{-\alpha}$ and $\frg_{-2\alpha},$ respectively, while $c^{-1}=4(m_{\alpha}+4m_{2\alpha}).$ Further, in view of \cite[Chapter 2, Theorem 6.1 (ii)]{H3}  we have
\begin{equation}
	\label{eq:baHv-1}
\exp(2\alpha H(v))=(1+c|X|^2)^2+4c|Y|^2
\end{equation}
and obtain
\begin{equation}
	\label{eq:baHv}
		\left[\frac{b^{\alpha}+b^{-\alpha}}{2}\right]^2=\frac{a^{-2\alpha}}{4}+\frac{a^{2\alpha}}{4}\exp(2\alpha H(v))+\frac{1}{2}(1+c|X|^2).
\end{equation}
Consequently, 
\begin{equation*}
	%\label{eq:baHvin}
	b^{\alpha}+b^{-\alpha}\ge a^{-\alpha} \qquad \textrm{and}\qquad b^{\alpha}+b^{-\alpha}\ge (\ta)^{\alpha},
\end{equation*}
and since $b^{\alpha}+b^{-\alpha}\ge 2$ we obtain 
 item 1).

Now we justify 2). Rearranging \eqref{eq:baHv}and using \eqref{eq:baHv-1} we see that 
\begin{align*}
\left[\frac{b^{\alpha}+b^{-\alpha}}{2}\right]^2&=\left[\frac{a^{-\alpha}}{2}+\frac{(\ta)^{\alpha}}{2}\right]^2+\frac{1}{2}\left[1+c|X|^2-\exp(\alpha H(v))\right]\\
&\le \left[\frac{a^{-\alpha}}{2}+\frac{(\ta)^{\alpha}}{2}\right]^2.
\end{align*}
	which implies 2).

To prove 3) we apply \eqref{e Phi} and items 1) and 2) obtaining
\begin{align*}
1-2(\ta)^{-2\alpha}\le \frac{\Phi[ (\ta^{\alpha})]}{\ta^{\alpha}} \le\frac{b^{\alpha}}{\ta^{\alpha}}\le \frac{\Phi(a^{-\alpha}+(\ta^{\alpha}))}{\ta^{\alpha}}\le 1+a^{-\alpha}(\ta)^{-\alpha}.
\end{align*}
 The proof of 4) proceeds similarly once we interchange $(\ta)^{\alpha}$ with $a^{-\alpha}$. This completes the proof of Lemma \ref{lem:CtoI}.
\end{proof}

\subsection{Estimates for $L^p$ norms of convolution operators}

The estimates presented below will be needed in the analysis of global parts of the considered operators in Section \ref{sec:glob}. 

Let $\Ga$ be a locally compact group with left Haar measure $\wrt y$. Throughout the paper $\Cvp{\Ga}$ will denote the space of bounded operators on $\lp{\Ga}$
that commute with left translations (equivalently the space of right convolutors), 
equipped with the operator norm on $\lp{\Ga}$. If $\phi$ is a  convolution (distributional) kernel of an operator $K_{\phi}f=f*_{\Ga} \phi$ in $\Cvp{\Ga}$ we write $\|\phi\|_{\Cvp{\Ga}}:=\|K_{\phi}\|_{\Cvp{\Ga}}.$ 

In this paper we deal with two types of groups $\Ga.$ Firstly, we consider $\Ga=\BS=\OBN \BA$, which is the semidirect product of the groups $\OBN$ and $\BA.$ In this case $a^{2\rho}dadv$ is a left Haar measure on $\Ga$ while $da dv$ is a right Haar measure on $\Ga.$ Secondly, in the case $\Ga=\BA$ the group is abelian, hence, unimodular with Haar measure $da$. Recall that we consider a multiplicative Haar measure $da$ on $\BA$. We let $D=a\partial_a$ be the multiplicative derivative, given, for $a=\exp(tH_0)$ by
$(D\phi) (a)=\partial_t (\phi(\exp(tH_0))).$

We denote by $\cM$ be the Mellin/Fourier transform on $\BA$ given, for $\phi\in L^1(\BA),$ by 
\begin{equation*}
	%\label{e MF}
\cM(\phi)(\la)=\int_{\BA}\phi(a)\,a^{-i\la}\,da,\qquad \la\in \fra^*.
\end{equation*}
The inverse Fourier transform is then defined, for $m\colon \fra^*\to \C,$ by 
\begin{equation}
	\label{e ft}
	\mF (m) (a)=\int_{a^*}m(\la)\,a^{i\la}\,d\la,\qquad a\in \BA,
\end{equation} 
and the inversion formula reads
\begin{equation}
	\label{e MFinv}
\phi(a)=C\int_{\fra^*}\cM(\phi)(\la)\,a^{i\la}\,d\la,
\end{equation}
where $C$ is a constant and $d\la$ denotes a Haar (Lebesgue) measure on the additive group $\fra^*.$ 

A function $m\colon \fra^*\to \C$ is called an $\lp{\BA}$ Fourier multiplier whenever the corresponding convolution operator with the kernel \eqref{e ft} is bounded on $\lp{\BA}.$      The space of $\lp{\BA}$ Fourier multipliers is denoted by $\Mp{\BA}$ and comes equipped with the operator norm on $L^p(\BA)$.  This norm coincides with $\|\phi\|_{\Cvp{\BA}}$ where  $m=\cM(\phi).$

Our next lemma says that multiplication by an appropriately smooth function $\phi$ does not change the norm of a convolution operator on $L^p(\BA).$ Similar properties hold for more general functions $\phi$ and on more general abelian groups, see \cite[Theorem 2.8]{MW} and \cite{Co}. Lemma \ref{lem Mel} is convenient because of the explicit control in terms of the $L^1(\BA)$ norms of $\phi$ and $D^2\phi$ that will be needed later.  
\begin{lem}
	\label{lem Mel}
	Let $\phi$ be a smooth function on $\BA$ such that $\phi \in L^1(\BA)$ and $D^2\phi\in L^1(\BA).$ Take $p\in [1,\infty).$ Then for any kernel $\kappa$ on $\BA$ we have
	\begin{equation*}
		%\label{e Mel}
		\|\phi \kappa\|_{\Cvp{\BA}}\lesssim (\|\phi\|_{L^1(\BA)}+\|D^2 \phi\|_{L^1(\BA)})\cdot \|\kappa\|_{\Cvp{\BA}}.
	\end{equation*}  
\end{lem} 
\begin{proof}
Using the inversion formula \eqref{e MFinv} we see that
\[
\phi(a)\kappa(a)=C\int_{\fra^*}\cM(\phi)(\la)[a^{i\la}\kappa(a)]\,d\la,
\]
Since $\|a^{i\la}\kappa(a)\|_{\Cvp{\BA}}=\|\kappa(a)\|_{\Cvp{\BA}},$ Minkowski's integral inequality implies that
	\[
	\|\phi \kappa\|_{\Cvp{\BA}}\leq \|\cM(\phi)\|_{L^1(\fra^*)} \|\kappa\|_{\Cvp{\BA}}.
	\]
	Now, integrating by parts twice we obtain $\|\cM(\phi)\|_{L^1(\fra^*)}\lesssim \|\phi\|_{L^1(\BA)}+\|D^2\phi\|_{L^1(\BA)}$ thus completing the proof.
\end{proof}

The $L^p(\BX)$ boundedness of the approximating global kernels in Section \ref{sec:glob} will be obtained from a transference result originating in \cite{CMW} and employed also in \cite{MW}.  We state a consequence  of  \cite[Theorem 3.3]{MW} that will be used in our paper.
\begin{pro} \label{t: Transference principle}
	Assume that $K$ is a kernel on $\BS$ such that $a^{2\rho/p}K\in \lu{\OBN;\Cvp{\BA}}.$ Then we have
	\begin{equation}
		\label{e tp1}
		\bignorm{K}{\Cvp{\BS}}
		\leq \bignorm{a^{2\rho/p}K}{\lu{\OBN;\Cvp{\BA}}},
	\end{equation}
	and, consequently,
	\begin{equation}
		\label{e tp2}
		\bignorm{K}{\Cvp{\BS}}
		\leq \int_{\OBN}\int_{\BA} a^{2\rho/p}|K(va)|\,da\,dv
	\end{equation}
\end{pro}

%\bc^{-1}(-\la)
\subsection{Spherical Fourier transform and spherical functions}

The spherical Fourier transform of an integrable function $g$ on $\BG$
is the function $\st g$, defined by
$$
\st g (\la)
= \int_{\BG} g(x) \, \vp_{-\la}(x) \wrt x
\qquad  \la \in \fra_\BC^*
$$
where $\vp_{\la}$ denotes the spherical functions on $\BG$.
Recall that $\st g$ is Weyl-invariant, which, in the rank one case means that it is an even function of $\la.$ The spherical Fourier transform extends to $\BK$--bi-invariant
tempered distributions on $\BG$ (see, for instance, \cite[Ch.~6.1]{GV}).

For ``nice'' $\BK$--bi-invariant functions $g$, the inversion formula is given by
\begin{equation*}
	%\label{e inv}
	g(x)
	= \int_{\fra^*} \st g(\la)\, \vp_{\la}(x) \, \wrt \nu(\la),
\end{equation*}
where the Plancherel measure $\nu$ is given by
$
\wrt \nu(\la)
= \bigmod{\bc(\la)}^{-2} \wrt\la,
$
and $\bc$ is the Harish-Chandra function. When applied to the kernel $k_B$ of the operator $B$ the inversion formula shows that
 \begin{equation}
 	\label{e inv km}
 	k_B(x)
 	= \int_{\fra^*} m_B(\la)\, \vp_{\la}(x) \, \wrt \nu(\la).
 \end{equation}
In order to ensure the convergence of various integrals appearing in the proof
we tacitly assume that $m_B$ is pre-multiplied by the factor $\exp(-\vep(\la(H_0))^2),$ $\vep>0.$ This implies that the corresponding kernel given by \eqref{e inv km} is a bounded smooth function. Since in the limit $\vep\to 0^+$ the parameter $\varepsilon>0$ vanishes from our bounds this operation has no effect on our results. 

 To simplify the exposition we often identify  $\C$ with $\fra_{\BC}^*$ using the map $\zeta\mapsto  \zeta \alpha.$ Via this identification the strip $T_p$ in $\C$  coincides with the tube $$\mathcal T_p:=\{\la\in \fra^*_{\BC}\colon |\Ima (\la(H_0))|<\rop\}$$
 in $\fra^*_{\BC}$ 
 and functions on $T_p$ are identified with functions on  $\mathcal T_p.$     Note that when $m_B$ is a bounded holomorphic function in $T_p$ then $m_B(\la\pm i\rop)$ exists as a non-tangential limit for almost every $\la\in \R$, see e.g. \cite[Chapter III]{SW}.

In the rank one case it is well known that there exist constant $C_j,$ $j=0,1,...,$ such that
\begin{equation} 
	\label{f: HCest}
	\bigmod{\partial^{j} \check\bc^{-1}(\la)}
	\leq  C_j \,  \big(1+|\la|\big)^{(n-1)/2-j}
	\qquad  \la: 0\leq \Im \la \leq |\rho|
\end{equation}
where $n$ denotes the dimension of $\BX$ (see, for instance, \cite[Lemma 4.2]{ST} and \cite[Appendix A]{I1}). For further reference we also note that for each $1<p<\infty$ the function $[\omega(\la+ i\rop)\bc(-\la-i\rop)]^{-1}$ satisfies Mikhlin-H\"ormander condition of arbitrary order on $\R,$ hence, it gives rise to a bounded operator on $\lp{\BA}.$ This observation follows from \eqref{eq:omdef} and \eqref{f: HCest}. \label{MHmw} Thus, in view of 
the decomposition
\begin{equation}
	\label{e ocmdecom}
\mbc(\la+ i\rop)
=[\omega(\la+ i\rop)\bc(-\la-i\rop)]^{-1}\, [\omega m_B](\la + i\rop)
\end{equation}
we see that if $[\omega m_B](\cdot + i\rop)$ is an $L^p(\BA)$ Fourier multiplier then also $\mbc(\cdot + i\rop)$ is an $L^p(\BA)$ Fourier multiplier.

At infinity, the celebrated Harish-Chandra expansion states that for $a$ in $\BA^+$ 
\begin{equation*} 
	%\label{f: HC expansion}
\vp_{\la}(a )
= a^{-\rho+i\la}\, \bc(\la)\, [1+\om(\la,a)]
      +a^{-\rho -i\la}\, \bc(-\la)\, [1+\om(-\la,a)],
\end{equation*}
where \begin{equation}
	\label{e ome}
	\om(\la,a) = \sum_{\ell=1}^{\infty}\, \Ga_{\ell}(\la)\, a^{-2\ell\al}
\end{equation}
and there exist constants $b_j$ such that the coefficients
$\Ga_\ell$ satisfy the following estimates \cite[Appendix A]{I1}
\begin{equation} 
\label{eq:GamEst}
\bigmod{\partial^j_{\la}\Ga_{\ell}(\la)}
\leq C_j\,  \ell^{b_j}\,  \big(1+|\Re \la|\big)^{-j},
\end{equation}
for $\ell\geq 1$ and $0\leq \Im \la \leq \rho$. Inequalities \eqref{eq:GamEst} imply the bound
\begin{equation} \label{f: estIonescu}
\bigmod{\partial^j_{\la}\om(\la,a)}+\bigmod{a\partial_a\partial^j_{\la}\om(\la,a)}
\leq C_j \, (1+|\Re \la|)^{-j}
\asymp (1+|\la|)^{-j},
\end{equation}
valid for all integers $j=0,1,\ldots$ whenever  $\al(\log a)\geq 1/2$ and $0\leq \Im \la\leq \rho.$

The expansions of spherical functions near the origin were established in \cite[Theorem 2.1]{ST}. We do not recall these expansions here as there are not explicitly needed in our paper.  This is because the local parts of the considered kernels will be treated by a reference to known results. 

For further reference we note that if  $m$ is a Weyl-invariant function on $\fra^*$ then
\begin{equation} \label{f: inversion infinity}
k_B(a)
= { 2} \int_{\fra^*} a^{i\la-\rho}\, \big[1+\om(\la,a) \big] \, \mbc (\la)  \wrt \la 
\qquad  a \in \BA^+,
\end{equation}
see \cite[Remark 2.1]{MW}.

\section{Estimates for local and global parts of the kernel}

\label{s:elg}

We split the kernel $k_B$ into several parts and treat each of them separately. 

 We start with distinguishing the local and global parts of $k_B.$  Let $\bC$ be a smooth even function on $\BR$
that is equal to~$1$ on $[-2,2]$, vanishes outside of the interval $[-4,4]$ and satisfies $0\leq \bC \leq 1$. 
Define a $\BK$--bi-invariant function $\eta$ on $\BG$ by 
$$
\eta(a) 
:= \bC \big(\al(\log a) \big)
\qquad  a \in \BA.  
$$ 
We decompose $k_B$ into its local and global parts via
\begin{equation*}
	%\label{eq:kloglo}
	k_B=k_{\loc}+k_{\glo}:=\eta k_B+(1-\eta)k_B.
\end{equation*}
This leads to the corresponding decomposition for the operator $B$, namely,
\begin{equation*}
	%\label{eq:bloglo}
	Bf=f*k_{\loc}+f*k_{\glo}:=B_{\loc}f+B_{\glo} f.
\end{equation*}

The local part $B_{\loc}$ will be easier to estimate. This is done in Section \ref{sec:loc}, mostly by referring to the literature. The underlying transference principle is that of Coifman and Weiss \cite{CW}.

 To estimate the global part $B_{\glo}$ we need to split it further. The underlying idea is that of Remark \ref{rem:informal}. The splittings are based according to which of the terms $a^{-\alpha}$ or $(a\exp(H(v)))^{\alpha}$ is dominant in Lemma \ref{lem:CtoI}. Then we apply a transference principle from \cite{CMW}. This is accomplished in Section \ref{sec:glob}.

\subsection{The local part}
\label{sec:loc}
It is not hard to see that under the assumptions of Theorem \ref{t: main} or Theorem \ref{t: main i} the operator $B_{\loc}$ is bounded on $L^p(\BX).$ In case of Theorem \ref{t: main} this will be deduced with the help of the following lemma.
\begin{lem}[{\cite{GMM}}]
	\label{lem:locm}
	Let $p\in (1,\infty)$ and assume that $m$ is a bounded holomorphic function in $T_p$. If $\la\mapsto m(\la\pm i\rop)$ is an $\lp{\BA}$ Fourier multiplier, then also $m(\la)$ is an $\lp{\BA}$ Fourier multiplier; moreover, 
	\begin{equation*}
		%\label{e locm}
		\|m\|_{\Mp{\BA}}\lesssim \| m(\cdot \pm i\rop)\|_{\Mp{\BA}}.
	\end{equation*}
	
\end{lem}
\begin{proof}
	The proof is essentially contained in \cite[p.\ 172-173]{GMM}, so we shall be brief. For every $z\in \C$ such that $-1\le \Re z\le 1$ we define the tempered distribution on $\R$ by
	$$\mF \nu_z(\cdot )=m(\cdot+i\rop z),$$
	in particular $\mF \nu_{\pm 1}=m(\cdot\pm i\rop).$ Moreover, for every $y\in \R$ we have
	\[
	\nu_{1+iy}=e^{i\rop y}\nu_1,\qquad \nu_{-1+iy}=e^{-i\rop y}\nu_{-1}.
	\]
	Hence, applying Stein's complex interpolation theorem to the analytic family of operators
	$R_z=f*\nu_z$ we conclude that $R_0$ is bounded on $\lp{\BA}.$ Since $\mF \nu_0(\cdot)=m(\cdot)$ the proof is completed.    
	
\end{proof}

Using Lemma \ref{lem:locm} and the results from \cite{GMM} we see that under the assumptions of Theorem \ref{t: main} the operator $B_{\loc}$ is bounded on $L^p(\BX).$ 

\begin{pro} \label{p: main loc} 
	
	Take $p\in (1,\infty)\setminus \{2\}$.  Then, under the assumptions of Theorem \ref{t: main} the operator $B_{\loc}$ is bounded on $\lp{\BX}$.
	
\end{pro}
\begin{proof}
	Since $m_B$ is Weyl-invariant we see that both $m(\cdot+i \rho_p)$ and $m(\cdot-i \rho_p)$ are $L^p(\BA)$ Fourier multipliers.    Applying Lemma \ref{lem:locm} with $m=\omega m_B$ we see that  $m$  is an $\lp{\BA}$ Fourier multiplier. Decompose
	\[
	m_B(\la) \bc^{-1}(-\la)=m(\la) [\omega(\la)\bc(-\la)]^{-1}.
	\]
	Recall that $[\omega(\la)\bc(-\la)]^{-1}$ is an $\lp{\BA}$ Fourier multiplier, cf.\ comments after \eqref{f: HCest} on p.\ \pageref{MHmw}. Thus also $m_B(\la) \bc^{-1}(-\la)$ is an $\lp{\BA}$ Fourier multiplier. Now, \cite[Proposition 3.2]{GMM} implies that $B_{loc}$ is bounded on $\lp{\BX}.$ 
\end{proof}

Using Lemma \ref{lem:locm} and \cite{I2} we now show that under the assumptions of Theorem \ref{t: main i} the operator $B_{\loc}$ is also bounded on $L^p(\BX).$
\begin{pro} \label{p: main loc i} 
	Take $p\in (1,\infty)\setminus \{2\}.$ Then, under the assumptions of Theorem \ref{t: main i} the operator $B_{\loc}$ is bounded on $\lp{\BX}$.
\end{pro}
\begin{proof}
	Condition (2) in Theorem \ref{t: main i} together with the fact that $m$ is holomorphic imply that $m_B$ satisfies a Mikhlin-H\"ormander condition on $\R,$ namely, 
\begin{equation}
		\label{eq:MihR}
		\bigmod{\partial^jm_B(\la)}
		\leq C\, |\la|^{-j}
		\qquad \la\in \R,\quad  j=0,\ldots, N,
	\end{equation} 
	where $N>(n+3)/2.$ It is known that if \eqref{eq:MihR} holds, then the local part $B_{\loc}$ is a bounded operator on $\lp{\BX},$ see e.g. \cite[Section 5, pp. 271-272]{ST}. This completes the proof.

\end{proof}

\subsection{The global part}
\label{sec:glob}
The global part of the kernel is a $\BK$--bi-invariant distribution. Thus, the $L^p(\BX)$ boundedness of $B_{\glo}$ is equivalent to the $L^p(\BS)$ boundedness of $Tf:=f*_{\BS} K,$ where $K\colon \BS \to \C$ is defined by $K(va)=k_{\glo}(va)$. Recall that $\BS=\OBN \BA.$   Let ${\bf H}$ be a smooth function on $\R$ which vanishes on $(-\infty,-2]$ and equals $1$ on $[2,\infty).$ Denote
\begin{equation*}
	\chi(a)={\bf H}(\alpha(\log a)),\qquad a\in \BA
\end{equation*}    
and define the splitting kernels for $K=K_1+K_2$ by
\begin{align}
\label{e K1}	K_1(va)&=\chi\left(a\exp\bigg(\frac12 H(v)\bigg)\right)K(va),\\
\label{e K2}	K_2(va)&=\left(1-\chi\left(a\exp\bigg(\frac12 H(v)\bigg)\right)\right)K(va),
\end{align}
and the corresponding operators by $B_1=f*K_1$ and  $B_2=f*K_2.$  The threshold $a\exp(2^{-1} H(v))$ comes from Lemma \ref{lem:CtoI}. Indeed, we have $$a^{-\alpha}\le (a\exp(H(v)))^{\alpha}\qquad\textrm{precisely when}\qquad ((a\exp(2^{-1} H(v))))^{\alpha}\ge 1.$$ Note that $K_1$ vanishes when $a^{\alpha}\le (\exp(-2^{-1}H(v)))^{\alpha}\cdot e^{-2},$ while $K_2$ vanishes when $a^{\alpha}\ge (\exp(-2^{-1}H(v)))^{\alpha}\cdot e^{2}.$

We will approximate the splitting kernels \eqref{e K1},\eqref{e K2} by the approximating kernels
\begin{align}
	\label{e aK1}	\tK_1(va)&=\chi\left(a\exp\bigg(\frac12 H(v)\bigg)\right)K(a\exp(H(v))),\\
	\label{e aK2}	\tK_2(va)&=\left(1-\chi\left(a\exp\bigg(\frac12 H(v)\bigg)\right)\right)K(a^{-1}).
\end{align}
Again, the form of these kernels is a consequence of Lemma \ref{lem:CtoI}, see also Remark \ref{rem:informal}. Namely, $[va]_+$ is of the order $a\exp(H(v))$ on the support of $K_1$ and it is of the order $a^{-1}$ on the support of $K_2.$

In what follows for $q\in (1,2)$ we define a kernel $\ka^q$ on $\BA$ by 
\begin{equation}
	\label{e Kq}
	\ka^q(a)=a^{2\rho/q}K(a),\qquad a\in \BA.
\end{equation}
\begin{pro}
	\label{pro approx trans}
	Let $p\in(1,2).$ Then  we have
	\begin{equation*}
		%\label{e taK1}
\| \tK_1\|_{\Cvp{\BS}}\le C_p \left\|\ka^p\right\|_{\Cvp{\BA}}
	\end{equation*} 
and
	\begin{equation*}
	%\label{e taK2}
\| \tK_2\|_{\Cvp{\BS}}\le C_p \left\|\ka^p\right\|_{\Cvp{\BA}}.
\end{equation*} 
\end{pro} 
\begin{proof}
Using transference inequality \eqref{e tp1} we obtain
\begin{equation*}
	\| \tK_1\|_{\Cvp{\BS}}\leq \int_{\OBN}\left\|a^{2\rho/p}\tK_1(va)\right\|_{\Cvp{\BA}}\,dv
\end{equation*}
and
\begin{equation*} 
	\| \tK_2\|_{\Cvp{\BS}}\leq \int_{\OBN}\left\|a^{2\rho/p}\tK_2(va)\right\|_{\Cvp{\BA}}\,dv,
\end{equation*}
where $dv$ denotes the Haar measure on the unimodular group $\OBN.$ Then we write
\begin{align*}
	\| \tK_1\|_{\Cvp{\BS}}&\leq \int_{\OBN}P(v)^{2/p}\left\|\left(a\exp H(v)\right)^{2\rho/p}\tK_1(va)\right\|_{\Cvp{\BA}}\,dv,\\
	\| K_2\|_{\Cvp{\BS}}&\leq \int_{\OBN}P(v)^{2/p}\left\|\left(a\exp H(v)\right)^{2\rho/p}\tK_2(va)\right\|_{\Cvp{\BA}}\,dv,
\end{align*}
with $P(v)$ being defined in \eqref{e Pv}. Since $(1+|H(v)|)P(v)^{2/p}\in L^1(\BS)$ when $p<2,$ our task reduces to proving  the estimates 
\begin{equation}
		\label{e tK1 1}
 \left\|\left(a\exp H(v)\right)^{2\rho/p}\tK_1(va)\right\|_{\Cvp{\BA}}\lesssim (1+|\alpha H(v)|) \left\|\ka^p\right\|_{\Cvp{\BA}}
\end{equation}
and 
\begin{equation}
		\label{e tK2 1}
 \left\|\left(a\exp H(v)\right)^{2\rho/p}\tK_2(va)\right\|_{\Cvp{\BA}}\lesssim (1+|\alpha H(v)|) \left\|\ka^p\right\|_{\Cvp{\BA}}
\end{equation}
for $v\in \OBN.$

We justify \eqref{e tK1 1} first. By multiplication invariance and \eqref{e aK1} we have
\begin{align*}
	&\left\|\left(a\exp H(v)\right)^{2\rho/p}\tK_1(va)\right\|_{\Cvp{\BA}}=\left\|\chi\left(a\exp\bigg(-\frac12 H(v)\bigg)\right) a^{2\rho/p}K(a)\right\|_{\Cvp{\BA}}\\
	&\le \left\|a^{2\rho/p}K(a)\right\|_{\Cvp{\BA}}+\left\|\left(1-\chi\left(a\exp\bigg(-\frac12 H(v)\bigg)\right)\right) a^{2\rho/p}K(a)\right\|_{\Cvp{\BA}}.
\end{align*}
Further, denoting 
\begin{equation}
	\label{e fipsi}
\varphi(a)=(1-\chi(a))a^{4\rho/p}\qquad \textrm{and}\qquad \psi_v(a)=-\chi\left(a\exp\bigg(-\frac12 H(v)\bigg)\right)+\chi(a),
\end{equation}
we obtain
\begin{align*}
	&\left\|\left(a\exp H(v)\right)^{2\rho/p}\tK_1(va)\right\|_{\Cvp{\BA}}\\
	&\le \left\|a^{2\rho/p}K(a)\right\|_{\Cvp{\BA}}+\left\|\varphi(a)a^{-2\rho/p}K(a)\right\|_{\Cvp{\BA}}+\left\|\psi_v(a)a^{2\rho/p}K(a)\right\|_{\Cvp{\BA}}.
\end{align*}
Since $K$ is even we have $\|a^{-2\rho/p}K(a)\|_{\Cvp{\BA}}=\|\ka^p\|_{\Cvp{\BA}}.$ Thus, by Lemma \ref{lem Mel} in order to justify \eqref{e tK1 1} it suffices to show that 
\begin{equation}
	\label{e vps}
\|\varphi\|_{L^1(\BA)}+\|D^2\varphi\|_{L^1(\BA)}\lesssim 1\qquad\textrm{and}\qquad \|\psi_v\|_{L^1(\BA)}+\|D^2\psi_v\|_{L^1(\BA)}\lesssim |\alpha H(v)|.
\end{equation}  

The first inequality in \eqref{e vps} is obvious once we note that $\varphi$ is a smooth function supported on the set of small $a.$ More precisely, $\varphi$ vanishes unless $a^{\alpha}\le e^2.$  To prove the second inequality we note that $\psi_v(\exp(tH_0))$ vanishes unless $-2\le t\le 2+2^{-1}\alpha(H(v)).$ Moreover, $|\psi_v(a)|\le 2$ and  $|D^2 \psi_v(a)|\lesssim 1,$ so that integrating in $a=\exp(tH_0)$ we obtain
 \[
  \|\psi_v\|_{L^1(\BA)}+\|D^2\psi_v\|_{L^1(\BA)}\lesssim \int_{-2}^{2+2^{-1}\alpha(H(v))}\,dt \lesssim (1+|\alpha(H(v))|),
 \]
 as desired.

 Now we move to the proof of \eqref{e tK2 1}. Using \eqref{e aK2} and multiplication invariance we have
 \begin{align*}
 	&\left\|\left(a\exp H(v)\right)^{2\rho/p}\tK_2(va)\right\|_{\Cvp{\BA}}\\
 	&= \left\|\left(a^2\exp H(v)\right)^{2\rho/p}\left(1-\chi\left(a\exp\bigg(\frac12 H(v)\bigg)\right)\right)a^{-2\rho/p}K(a^{-1})\right\|_{\Cvp{\BA}}\\
 	&=\left\|\varphi\left(a\exp\bigg(\frac12 H(v)\bigg)\right)a^{-2\rho/p}K(a^{-1})\right\|_{\Cvp{\BA}},
 \end{align*}
where $\varphi$ was defined in \eqref{e fipsi}. Since the  $L^1(\BA)$ norms of $$\varphi\bigg(a \exp\bigg(\frac12 H(v)\bigg)\bigg)\qquad \textrm{ and }\qquad   D^2\bigg[\varphi\bigg(\cdot \exp\bigg(\frac12 H(v)\bigg)\bigg)\bigg](a)$$ are independent of $v\in \OBN$, applying Lemma \ref{lem Mel} and \eqref{e vps} we reach \eqref{e tK2 1}. This completes the proof of the proposition.
\end{proof}

We finish this section with a result that allows us to control the operators corresponding to the differences $K_j-\tK_j,$ $j=1,2.$ Here Lemma \ref{lem:CtoI} will be crucial.  For a function $\kappa\colon\BA\to \C$ we define the Lipschitz norm
\begin{equation}
	\label{e lip}
	\|\kappa\|_{\Lip}=\sup_{a,b,\in \BA}\frac{|\kappa(a)-\kappa(b)|}{|\alpha(\log a) -\alpha(\log b)|}.
\end{equation}

\begin{pro}
	\label{pro diff trans}
	Let $p\in(1,2)$. Take $q\in [p,2]$ such that $|\rho|(1/p-1/q)<1$ and assume that $\|\kappa^q\|_{\Lip}<\infty.$ Then we have
	\begin{equation}
		\label{e dK1}
		\| K_1-\tK_1\|_{\Cvp{\BS}}\le C_{p,q} \|\kappa^q\|_{\Lip}
	\end{equation} 
	and
	\begin{equation}
		\label{e dK2}
		\|K_2- \tK_2\|_{\Cvp{\BS}}\le C_{p,q}\|\kappa^q\|_{\Lip}.
	\end{equation} 
In particular, if $|\rho|<2,$ then \eqref{e dK1}, \eqref{e dK2} hold with $q=2.$ 
\end{pro} 
\begin{proof}
Clearly, if $|\rho|<2$ then $|\rho|(1/p-1/2)<1$ for all $p\in(1,2).$ Thus, in this case \eqref{e dK1} and \eqref{e dK2} hold with $q=2.$ 

Note that since $\ka^q$ is Lipschitz and $\ka^q(\id)=0$ we have
\begin{equation}
	\label{e kqinf}|\ka^q(a)|\leq \|\ka^q\|_{\Lip}|\alpha(\log a)|,\qquad a\in \BA.
\end{equation}
Since $K$ is even we also  have $$a^{-2\rho/q}\ka^q(a)=K(a)=K(a^{-1})=a^{2\rho/q}\ka^q(a^{-1}),$$ so that
\begin{equation}
	\label{e Kkaq}
	|K(a)|\lesssim \|\ka^q\|_{\Lip}\qquad \textrm{and}\qquad \|K\|_{\Lip}\lesssim  \|\ka^q\|_{\Lip};
\end{equation}  
in the second inequality above we also used the fact that $K$ is supported away from the identity element $e_{\BA}.$ 

 For $v\in \OBN$ and $a\in \BA$ we denote \[
b=[va]_+,\qquad \ta=a\exp\big( H(v)\big),\qquad \tc=a\exp\big(2^{-1} H(v)\big).\]
Note that then we have $$aa_v=\tc^{2},$$ so that $$a^{\alpha}\le \tc^{\alpha}\le \ta^{\alpha}.$$

{\bf We start with the proof of \eqref{e dK1}}. Here the kernel we need to estimate is
\begin{align*}
	K_1(va)-\tK_1(va)=\chi(\tc)\left[b^{-2\rho/q}\ka^q(b)-\ta^{-2\rho/q}\ka^q(\ta)\right].
\end{align*}  
By the transference inequality \ref{e tp2} it is to show that
\begin{equation}
	\label{e K1-tK1 goal}
	\int_{\OBN}\int_{\BA} a^{2\rho/p}|(K_1-\tK_1)(va)|\,da\,dv \lesssim\|\ka^q\|_{\Lip}.
\end{equation}
If $\chi(\tc)\neq 0,$ then we have $\tc^{\alpha}\ge e^{-2}.$ Thus, if $(\ta)^{\alpha}<2$ then  $va$ belongs to the set  
\[\{va\in \BS \colon  e^{-2}\exp(-2^{-1}\alpha H(v))\le a^{\alpha}\le 2 \exp(-\alpha H(v))\}.\] 
This set is included in
\[
E_1:=\{va\in \BS\colon \exp(2^{-1}\alpha H(v))\le 2 e^2,\quad 2^{-1}e^{-2}\le a^{\alpha}\le 2\},
\] 
%Below we also use that b^{\alpha}\ge 2. 
which is compact, hence, of finite measure in $\BS.$    
Using \eqref{e kqinf} we thus see that
	$|K_1(va)-\tK_1(va)|\Ind{\ta^{\alpha}<2}\lesssim  \|\ka^q\|_{\Lip}\Ind{E_1}(va),$ hence, also 
	\[
	\int_{\OBN}\int_{\BA} a^{2\rho/p}|K_1(va)-\tK_1(va)|\Ind{\ta^{\alpha}<2}\,da\,dv\lesssim \|\ka^q\|_{\Lip}.
	\] 
	
Thus, we are left with estimating $|K_1(va)-\tK_1(va)|\Ind{\ta^{\alpha}\ge 2}.$ 
Using Lemma \ref{lem:CtoI} item 3) we see that
\begin{equation*}
-(\ta)^{-2\alpha}\lesssim\log(1-2(\ta)^{-2\alpha})\le \alpha(\log b) -\alpha(\log \ta) \le \log (1+\tc^{-2\alpha})\le \tc^{-2\alpha},
\end{equation*}
so that
\begin{equation}
	\label{e bta}
	| \alpha(\log b) -\alpha(\log \ta)|\lesssim (\ta)^{-2\alpha}+(\tc)^{-2\alpha}\lesssim (\tc)^{-2\alpha}.
\end{equation}
Since $b^{\alpha}\ta^{-\alpha}=\exp( \alpha(\log b) -\alpha(\log \ta))$ the above inequality implies that
 \begin{equation}
 	\label{e btar}
 1-C\tc^{-2\alpha}\le  \frac{b^{2\rho/q}}{\ta^{{2\rho/q}}}\le 1+C\tc^{-2\alpha},
 \end{equation}  
where $C$ is a constant depending only on $|\rho|$ and $q.$

Thus, decomposing 
\begin{align*}
	&K_1(va)-\tK_1(va)=\chi(\tc)\left[b^{-2\rho/q}\ka^q(b)-\ta^{-2\rho/q}\ka^q(\ta)\right]\\
	&=\chi(\tc) b^{-2\rho/q}\left[\ka^q(b)-\ka^q(\ta)\right]+\chi(\tc)b^{-2\rho/q}\ka^q(\ta)(1-b^{2\rho/q}\ta^{-2\rho/q})
\end{align*}   
and using the definition \eqref{e lip} together with \eqref{e kqinf},  \eqref{e bta}, and \eqref{e btar}, we estimate
\begin{equation}
	\label{e K1-tK1}
	a^{2\rho/p}|K_1(va)-\tK_1(va)|\Ind{\ta^{\alpha}\ge 2}\le a^{2\rho/p}\|\kappa^q\|_{\Lip}\cdot \chi(\tc)  b^{-2\rho/q}\tc^{-2\alpha}(1+|\alpha(\log \ta)|)
\end{equation}
and further
\begin{equation}
		\label{e K1-tK1'}
		\begin{split}
		&a^{2\rho/p}|K_1(va)-\tK_1(va)|\Ind{\ta^{\alpha}\ge 2}\\
		&\lesssim \|\kappa^q\|_{\Lip}\cdot \chi(\tc)  a^{2\rho(1/p-1/q)-2\alpha}\exp((-2\rho/q-\alpha)H(v))(1+|\alpha(\log a)|+|\alpha(H(v))|).
		\end{split}
\end{equation}
By our assumption $2|\rho|(1/p-1/q)-2<0$ we see that
\[
\int_{a\colon a^{\alpha}\ge 2} (1+|\alpha(\log a)|)a^{2\rho(1/p-1/q)-2\alpha}\,da<\infty.
\]
Hence, using \eqref{e K1-tK1'} in view of \eqref{e Pv'} we obtain
\begin{align*}
	&\int_{\OBN}\int_{\BA}a^{2\rho/p}|K_1(va)-\tK_1(va)|\Ind{a^{\alpha}\ge 2}\,dv\,da \\
	&\le \|\kappa^q\|_{\Lip} \int_{\OBN} (1+|\alpha(H(v))|)\exp((-2\rho/q-\alpha)H(v))\,dv \lesssim  \|\kappa^q\|_{\Lip}.
\end{align*}

In order to prove \eqref{e K1-tK1 goal} it remains to control  \[\int_{\OBN}\int_{ \BA}a^{2\rho/p}|K_1(va)-\tK_1(va)|\Ind{\ta^{\alpha}\ge 2}\Ind{a^{\alpha}<2}\,da\,dv.\] Assume first that $q<2.$  We use \eqref{e K1-tK1}  and estimate $\tc^{-2\alpha}\lesssim 1$ which, together with \eqref{e Pv'},  leads to   
\begin{align*}
	&\int_{\OBN}\int_{\BA}a^{2\rho/p}|K_1(va)-\tK_1(va)|\Ind{\ta^{\alpha}\ge 2}\Ind{a^{\alpha}<2}\,dv\,da\\
	& \lesssim \|\kappa^q\|_{\Lip} \int_{a\colon a^{\alpha}<2}(1+|\alpha(\log a)|)a^{2\rho(1/p-1/q)}\,da\int_{\OBN} (1+|\alpha(H(v))|)\exp((-2\rho/q)H(v))\,dv \\
	&\lesssim  \|\kappa^q\|_{\Lip}.
\end{align*}
We are left with the case $q=2.$ Take $\varepsilon\in(0,1)$ such that $2|\rho|(1/p-1/2)>\varepsilon.$ Then we have
\[
\int_{a\colon a^{\alpha}< 2} (1+|\alpha(\log a)|)a^{2\rho(1/p-1/2)-\varepsilon\alpha}\,da<\infty.
\]
Note that for $1\lesssim \tc^{\alpha}$ we have $\tc^{-2\alpha}\lesssim \tc^{-\varepsilon \alpha}$. Hence, using again  \eqref{e K1-tK1}
 together with \eqref{e btar}  and \eqref{e Pv'} we obtain
\begin{align*}
	&\int_{\OBN}\int_{\BA}a^{2\rho/p}|K_1(va)-\tK_1(va)|\Ind{\ta^{\alpha}\ge 2}\Ind{a^{\alpha}<2}\,dv\,da \\
	&\lesssim\|\kappa^2\|_{\Lip} \int_{a\colon a^{\alpha}< 2} (1+|\alpha(\log a)|)a^{2\rho(1/p-1/2)-\varepsilon\alpha}\,da \\
	&\hspace{1.5cm}\times  \int_{\OBN} (1+|\alpha(H(v))|)\exp((-\rho-2^{-1}\varepsilon\alpha)H(v))\,dv \\
	&\lesssim \|\kappa^2\|_{\Lip}. 
\end{align*}

This completes the proof \eqref{e K1-tK1 goal} and thus also the proof of \eqref{e dK1}.

{\bf Now we move to the proof of \eqref{e dK2}}. Here the kernel we need to estimate is
\begin{align*}
	K_2(va)-\tK_2(va)=(1-\chi(\tc))\left[b^{-2\rho/q}\ka^q(b)-a^{2\rho/q}\ka^q(a^{-1})\right].
\end{align*}  
The argumentation is analogous to the proof of \eqref{e dK1}. The main difference is the use of Lemma \ref{lem:CtoI} item 4) instead of item 3) in a relevant place. We present the argument for the sake of completeness. 

By the transference inequality \ref{e tp2} it is enough to show that
\begin{equation}
	\label{e K2-tK2 goal}
	\int_{\OBN}\int_{\BA} a^{2\rho/p}|(K_2-\tK_2)(va)|\,da\,dv \lesssim\|\ka^q\|_{\Lip}.
\end{equation}
If $1-\chi(\tc)\neq 0,$ then $\tc^{\alpha}\le e^{2}.$ Thus, if $a^{-\alpha}<2$ then $va$ belongs to the compact set  
\[E_2:=\{va\in \BS \colon \exp(2^{-1}\alpha H(v))\le 2 e^2,\quad  2^{-1}\le a^{\alpha}\le e^2 \}.\]  
Using \eqref{e kqinf} and \eqref{e Kkaq} we thus see that
$|K_2(va)-\tK_2(va)|\Ind{a^{-\alpha}<2}\lesssim  \|\ka^q\|_{\Lip}\Ind{E}(va),$ hence, also 
\[
	\int_{\OBN}\int_{\BA} a^{2\rho/p}|K_1(va)-\tK_1(va)|\Ind{a^{-\alpha}<2}\,da \,dv\lesssim \|\ka^q\|_{\Lip}.
\] 
Thus, we are left with estimating \[	\int_{\OBN}\int_{\BA} a^{2\rho/p}|K_1(va)-\tK_1(va)|\Ind{a^{-\alpha}\ge 2}\,da \,dv.\] 
Using Lemma \ref{lem:CtoI} item 4) we see that
\begin{equation*}
	-a^{2\alpha}\lesssim\log(1-2a^{2\alpha})\le \alpha(\log b) -\alpha(\log a^{-1}) \le \log (1+\tc^{2\alpha})\le \tc^{2\alpha},
\end{equation*}
so that
\begin{equation}
	\label{e bta'}
	| \alpha(\log b) -\alpha(\log a^{-1})|\lesssim a^{2\alpha}+(\tc)^{2\alpha}\lesssim (\tc)^{2\alpha}.
\end{equation}
In analogy with \eqref{e btar} Lemma \ref{lem:CtoI} item 4) also gives
\begin{equation}
	\label{e btar'}
	1-C(\tc)^{2\alpha}\le  \frac{b^{2\rho/q}}{a^{-2\rho/q}}\le 1+C(\tc)^{2\alpha},
\end{equation}  
where $C$ is a constant depending only on $|\rho|$ and $q.$

Thus, decomposing 
\begin{align*}
	&K_2(va)-\tK_2(va)=(1-\chi(\tc))\left[b^{-2\rho/q}\ka^q(b)-a^{2\rho/q}\ka^q(a^{-1})\right]\\
	&=(1-\chi(\tc)) b^{-2\rho/q}\left[\ka^q(b)-\ka^q(a^{-1})\right]+(1-\chi(\tc))b^{-2\rho/q}\ka^q(a^{-1})(1-b^{2\rho/q}a^{2\rho/q})
\end{align*}   
and using \eqref{e kqinf} and \eqref{e bta'},\eqref{e btar'} we estimate
\begin{equation}
		\begin{split}
\label{e K2-tK2} &a^{2\rho/p}|K_2(va)-\tK_2(va)|\Ind{a^{-\alpha}\ge 2}\\
&\lesssim a^{2\rho/p}\|\kappa^q\|_{\Lip}\cdot (1-\chi(\tc))  b^{-2\rho/q}\tc^{2\alpha}(1+|\alpha(\log a^{-1})|)\\
  &\lesssim \|\kappa^q\|_{\Lip}\cdot (1- \chi(\tc))  a^{2\rho(1/p+1/q)+2\alpha}\exp(\alpha H(v))(1+|\alpha(\log a)|).
 \end{split}
\end{equation}
By our assumptions $1/p+1/q >1,$ so that there exists $\varepsilon\in (0,1/2)$ such that 
\begin{equation}
	\label{e rhopq}
	|\rho|(1/p+1/q)-\varepsilon >|\rho|(1+\varepsilon),
	\end{equation}
	  Since $(1-\chi(\tc))$ implies $\tc^{\alpha}\le e^{2}$ using \eqref{e K2-tK2} we obtain
\begin{align*}
	&\int_{a\colon a^{-\alpha}\ge 2} a^{2\rho/p}|K_2(va)-\tK_2(va)|\,da\\
	&
\lesssim_{\varepsilon} \|\kappa^q\|_{\Lip}\exp(\alpha H(v))\int_{a\colon a^{\alpha}\le e^2 \exp(-2^{-1}\alpha H(v))} a^{2\rho(1/p+1/q)+(2-2\varepsilon)\alpha}\,da\\
&\lesssim\|\kappa^q\|_{\Lip}\exp\left(-(\rho(1/p+1/q)-\varepsilon \alpha) H(v)\right) 
\end{align*}
Recalling \eqref{e rhopq} and using \eqref{e Pv'} we see that $\exp\left(-(\rho(1/p+1/q)-\varepsilon \alpha) H(v)\right) $ is in $L^1(\OBN).$ Therefore, we reach

\[
\int_{\OBN}\int_{\BA} a^{2\rho/p}|K_1(va)-\tK_1(va)|\Ind{a^{-\alpha}\ge 2}\,da \,dv\lesssim \|\ka^q\|_{\Lip}.
\] 

This completes the proof of \eqref{e K2-tK2 goal} and thus also the proof of \eqref{e dK2}.  The proof of Proposition \ref{pro diff trans} is also finished.

\end{proof}
\section{Proofs of our main results - Theorems \ref{t: main} and \ref{t: main i}}

	\label{s pf mult}
	
	In proving these theorems we will use the inversion formula \eqref{f: inversion infinity} which implies that
	\[
	K(a)={ 2} (1-\eta(a))\int_{\fra^*} a^{i\la-\rho}\, \big[1+\om(\la,a) \big] \, \mbc (\la)  \wrt \la 
	\qquad  a \in \BA^+;
	\] 
	recall that $\eta$ is a smooth cut-off function supported near the identity element $e_{\BA}.$  
	Denote $\ap=a$ if $a\in \BA_+$ and $\ap=a^{-1}$ if $a\in \BA^-.$ Since $K$ and $\eta$ are even we thus see that for $a\in \BA$   
	\begin{equation}
		\label{e kaqpinv0} 
		\kappa^q(a)= 
		2a^{2\rho/q}(1-\eta(\ap))\int_{\fra^*} \ap^{i\la-\rho}\, \big[1+\om(\la,\ap) \big] \, \mbc (\la)  \wrt \la.
	\end{equation}
	where $\ka^q$ is given for $q\in[p,2]$ by \eqref{e Kq}.  Using  \eqref{e kaqpinv0} and the change of variable $\la \to \la+i\roq$ for $a\in \BA^+$ we also see that
	\begin{equation}
		\label{e kaqpinv} 
		\kappa^q(a)=2(1-\eta(a))\int_{\fra^*} a^{i\la}\, \big[1+\om(\la+i\roq,a) \big] \, \mbc (\la+i\roq)  \wrt \la.
	\end{equation}
	
	Recall that $\om$ was defined by \eqref{e ome} as
	\begin{equation}
		\label{e ome'}
		\om(\la,a) = \sum_{\ell=1}^{\infty}\, \Ga_{\ell}(\la)\, a^{-2\ell\al}
	\end{equation}
	and denote $\Gamma_0\equiv 1.$ Inequality \eqref{eq:GamEst} (see also the proof of \cite[Lemma 3.3]{GMM}) implies that for all $0\le \de \le 1$ there is a constant $d>0$  such that
	\[
	|\la \partial_{\la}^{j} \Gamma_{\ell}(\la+i\de|\rho|)|\lesssim \ell^d,\qquad j=0,1,\qquad \ell=0,1,2,...
	\]
	uniformly in $\la \in \R.$ By the Mikhlin-H\"ormander multiplier theorem we thus see that
	\begin{equation}
		\label{e GjMpnorm}
		\|\Gamma_{\ell}(\cdot+i \de|\rho|)\|_{\Mp{\BA}}\lesssim \ell^d,
	\end{equation}
	for each fixed $0\le \de \le 1.$

	For further reference we define the functions
	\begin{equation}
		\label{e epm}
		\begin{split}
		\eta_{+}(a)&:=2(1-\eta(a))\Ind{\BA^{+}}(a),\qquad \eta_{-}(a):=2(1-\eta(a^{-1}))a^{2\rho/p}\Ind{\BA^{-}}(a),\\
		\eta_{\ell, +}(a)&:=\eta_{+}(a)a^{-2\ell \alpha},\qquad \eta_{\ell, -}(a):=\eta_{-}(a)a^{2\ell \alpha},\qquad \ell=0,1,2,\ldots.
	\end{split}
\end{equation}
	Note that $\eta_{\pm}$ and $\eta_{\ell, \pm}$ are smooth functions on $\BA.$ Further, since $(1-\eta(\ap))$ vanishes if $|\alpha(\log a)|<2$  a short calculation shows that
	\begin{equation}
		\label{e epmles}
	\|\eta_{\ell, \pm}\|_{L^1(\BA)}+\|D^2 \eta_{\ell, \pm}\|_{L^1(\BA)}\lesssim e^{-\ell},\qquad \ell=0,1,2,... ,
	\end{equation}
so that 
	\begin{equation*}
	%\label{e epmes}
	\|\eta_{ \pm}\|_{L^1(\BA)}+\|D^2 \eta_{ \pm}\|_{L^1(\BA)}\lesssim 1.
\end{equation*}
The above inequalities will be important for applications of Lemma \ref{lem Mel} later on.

Throughout the proof of Theorem \ref{t: main} and \ref{t: main i} we abbreviate
\begin{equation}
	\label{e tmbc}
	\tm=\check\bc^{-1}m_B.
\end{equation}

We start with the proof of Theorem \ref{t: main}.

\subsection{Proof of  Theorem \ref{t: main}}. By Proposition \ref{p: main loc} it suffices to estimate the global part $K=k_{\glo}$ under the assumptions of Theorem \ref{t: main}. In Section \ref{sec:glob} the kernel $K$ was split as 
\[
K=K_1+K_2=\tK_1+\tK_2+(K_1-\tK_1)+(K_2-\tK_2),
\]
where $K_1$ and $K_2$ are the splitting kernels \eqref{e K1}, \eqref{e K2}, while  $\tK_1$ and $\tK_2$ are the approximating kernels \eqref{e aK1}, \eqref{e aK2}. Hence, by Propositions \ref{pro approx trans} and \ref{pro diff trans} Theorem \ref{t: main} will be proved if we show the two propositions below. 
\begin{pro}
	\label{pro kap}
	Assume that $m_B$ satisfies assumption 1) of Theorem \ref{t: main}. 
	 Namely, let $m_B$ be a bounded Weyl-invariant holomorphic function in $T_p$ such that
$[\omega m_B](\cdot + i\rop) $ is an $L^p(\BA)$ Fourier multiplier. 
	Then we have $\|\ka^p\|_{\Cvp{\BA}}<\infty.$ 
\end{pro}
\begin{pro}
	\label{pro kaq}
	Assume that $m_B$ satisfies assumption 2) of Theorem \ref{t: main}. Namely,
	 let $q\in [p,p']$ be such that $|\rho| |1/p-1/q|<1$ and
	\begin{equation}
		\label{e kaq asum}
	\bigmod{\partial^jm_B(\zeta)}
	\leq C\, \Big(\min \big[\mod{\zeta-i\roq},\mod{\zeta+i\roq}\big]\Big)^{-j}
	\qquad  j=0,\ldots, N,\quad \zeta \in \TWq. \end{equation}
	Then we have $\|\ka^q\|_{\Lip}<\infty.$ 
	\end{pro}

\begin{proof}[Proof of Proposition \ref{pro kap}]
Using the decomposition \eqref{e ocmdecom} we see that it suffices to prove that
\begin{equation}
	\label{e kpmc}
	\|\ka^p\|_{\Cvp{\BA}}\lesssim \| \tm (\cdot + i\rop) \|_{\Mp{\BA}}.
\end{equation}
Then, using  \eqref{e kaqpinv} (for $a\in \BA^+$) and \eqref{e kaqpinv0} (for $a\in \BA^-$) together with \eqref{e ome'} we obtain the decomposition \begin{equation}
	\label{e kpdec}
	\ka^p=\sum_{l=0}^{\infty}\ka_{\ell}^{\pm}\end{equation} where, for $a\in \BA$ we define 
\begin{align*}
\ka_{\ell}^{+}(a)&= \eta_{\ell, +}\int_{\fra^*} a^{i\la}\,  \Gamma_{\ell}(\la+i\rop) \, \tm(\la+i\rop)  \wrt \la\\
\ka_{\ell}^{-}(a)&= \eta_{\ell, -}\int_{\fra^*} a^{-i\la}\,  \Gamma_{\ell}(\la) \, \tm(\la)  \wrt \la
\end{align*}

Therefore, using Lemma \ref{lem Mel}, \eqref{e epmles}, and \eqref{e GjMpnorm} we reach
\[
\|\ka_{\ell}^{\pm}\|_{\Cvp{\BA}}\lesssim \ell^d e^{-\ell}\|\tm(\cdot +i\rop)\|_{\Mp{\BA}},\qquad \ell=0,1,2,\ldots.
\]
Summing the above estimate over $\ell$  and recalling \eqref{e kpdec} we obtain \eqref{e kpmc}. This completes the proof of Proposition \ref{pro kap}.

\end{proof}

\begin{proof}[Proof of Proposition \ref{pro kaq}]

By the mean value theorem, to prove that $\|\ka^q\|_{\Lip}<\infty$ it suffices to show that $D \ka^q(a)$ is a bounded function on $\BA$; recall that $D=a\partial_a$ is the multiplicative derivative on $\BA.$ Since $\ka^q$ is supported away from $e_{\BA}$ it is enough to estimate $D \ka^q(a)$ separately for $a\in \BA^+$ and $a\in \BA^-.$  

We start with $a\in \BA^+$ in which case we utilize \eqref{e kaqpinv}. To compute the $D$ derivative we first integrate by parts $N$ times in $\la$ and then use Leibniz rule 
\begin{equation}
\label{e kaqform} 
\begin{split}&D\ka^q(a)\\
	=&2(-i)^N D\left(\frac{(1-\eta(a))}{(\log a)^N}\right)\int_{\fra^*} a^{i\la}\,  \partial_{\la}^N\left(\big[1+\om(\la+i\roq,a) \big] \, \tm (\la+i\roq )\right)   \wrt \la\\
&+2(-i)^N\frac{(1-\eta(a))}{(\log a)^N} \int_{\fra^*} a^{i\la}\, i\la\,\partial_{\la}^N\left(\big[1+\om(\la+i\roq,a) \big] \, \tm (\la+i\roq)\right)\\
&+ 2(-i)^N\frac{(1-\eta(a))}{(\log a)^N}\int_{\fra^*} a^{i\la}\,D\partial_{\la}^N\left(\big[1+\om(\la+i\roq,a) \big] \, \tm (\la+i\roq)\right).
\end{split}
\end{equation}
At this point we use our assumption \eqref{e kaq asum}. Since $N\ge (n+3)/2$ a calculation using \eqref{f: HCest}, \eqref{f: estIonescu} and \eqref{e kaq asum} gives
\[
| D^j\partial_{\la}^N\left(\big[1+\om(\la+i\roq,\ap) \big] \, \tm (\la+i\roq)\right)|\lesssim (1+|\la|)^{-3},\qquad j=0,1.
\]
Since $(1-\eta(a))(\log \ap)^{-N}$ is a smooth function equal to $(\log \ap)^{-N}$ for $|\alpha(\log a)|\ge 4$ coming back to  \eqref{e kaqform} we obtain 
\[
|D\ka^q(a)|\lesssim \int_{\fra^*} (1+|\la|)^{-2} \wrt \la\lesssim 1,\qquad a\in \BA^+.
\]

It remains estimate  $D\ka^q(a)$ when $a\in \BA^-.$ Here we use \eqref{e kaqpinv0} and again integrate by parts $N$ times in $\la.$ This gives, for $a\in \BA^-,$ 
\begin{align*}
	\kappa^q(a)=2i^Na^{\rho(2/q+1)}\frac{(1-\eta(a^{-1}))}{(\log a)^N}\int_{\fra^*} a^{-i\la}\, \partial_{\la}^N\left(\big[1+\om(\la,a^{-1}) \big] \, \tm (\la)\right)  \wrt \la.
\end{align*}
Using the above formula we calculate $D\kappa^q(a)$ with the aid of Leibniz rule and proceed similarly as we did when estimating  \eqref{e kaqform}. Here we need the assumption \eqref{e kaq asum} together with \eqref{f: HCest} and \eqref{f: estIonescu} (for $\Ima \la=0$). At this place it is important that $a^{\rho(2/q+1)}(1-\eta(a^{-1}))(\log a)^{-N}$ and its $D$ derivative are smooth functions that vanish for a close to the identity element $e_{\BA}.$   Then a calculation gives
\[
|D\ka^q(a)|\lesssim \int_{\fra^*} (1+|\la|)^{-2} \wrt \la\lesssim 1,\qquad a\in \BA^-.
\]

Thus we completed the proof of Proposition \ref{pro kaq}.

\end{proof}

Having completed the proof of Theorem \ref{t: main} we now prove Theorem \ref{t: main i}.

\subsection{Proof of  Theorem \ref{t: main i}} By Proposition \ref{p: main loc i} it suffices to estimate the global part $K=k_{\glo}$ under the assumptions of Theorem \ref{t: main i}. Using again the splitting
\[
K=K_1+K_2=\tK_1+\tK_2+(K_1-\tK_1)+(K_2-\tK_2),
\]
and Propositions \ref{pro approx trans} and \ref{pro diff trans} (with $q=p$) Theorem \ref{t: main} is reduced to the two propositions below. 
\begin{pro}
	\label{pro kap i}
	Assume that $m_B$ satisfies the assumptions 1) and 2) of Theorem \ref{t: main i}.
	Then we have $\|\ka^p\|_{\Cvp{\BA}}<\infty.$ 
\end{pro}
\begin{pro}
	\label{pro kaq i}
Assume that $m_B$ satisfies the assumptions 1) and 2) of Theorem \ref{t: main i}. 
	Then we have $\|\ka^p\|_{\Lip}<\infty.$ 
\end{pro}
In the proofs of the above propositions we shall need the splitting
\begin{equation}
	\label{e kapsplit}
			\ka^p(a)=\ka_1^p(a)+\ka_2^p(a)
\end{equation}
where
\begin{align*}
	&\ka_1^p(a)\\
	&=\begin{cases}
			\eta_+(a)\int_{\fra^*} a^{i\la}\bC(\la)\, \big[1+\om(\la+i\rop,a) \big] \, \tm (\la+i\rop)  \wrt \la,\quad a\in \BA^+\\
		\eta_-(a)\int_{\fra^*} a^{-i\la}\bC(\la)\, \big[1+\om(\la,a^{-1}) \big] \, \tm (\la)  \wrt \la,\quad a\in \BA^-.
	\end{cases}
\end{align*}
and
\begin{align*}
	&\ka_2^p(a)\\
	&=\begin{cases}
	\eta_+(a)\int_{\fra^*} a^{i\la}(1-\bC)(\la)\, \big[1+\om(\la+i\rop,a) \big] \, \tm (\la+i\rop)  \wrt \la,\quad a\in \BA^+\\
			\eta_-(a)\int_{\fra^*} a^{-i\la}(1-\bC)(\la)\, \big[1+\om(\la,a^{-1}) \big] \, \tm (\la)  \wrt \la,\quad a\in \BA^-.
	\end{cases}
\end{align*}
Recall that $\eta_{\pm}$ are defined by \eqref{e epm} while  $\bC$ is a smooth even function on $\BR$
that is equal to~$1$ on $[-2,2]$ and vanishes outside the interval $[-4,4]$. The splitting \eqref{e kapsplit} follows from \eqref{e kaqpinv0} (for $a\in \BA^-$) and \eqref{e kaqpinv} (for $a\in \BA^+$) with $q=p.$

\begin{proof}[Proof of Proposition \ref{pro kap i}]
	By \eqref{e kapsplit} it is enough to estimate $\|\ka_j^p\|_{\Cvp{\BA}},$ $j=1,2.$ 
	
	We start with $\ka_2^p.$ Integrating by parts $N$ times in $\lambda$ we see that
	\begin{align*}
		&\ka_2^p(a)\\
		&=\frac{\eta_+(a)(-1)^N}{(\log a)^N}\int_{\fra^*} a^{i\la}\partial_{\la}^N\left((1-\bC)(\la)\, \big[1+\om(\la+i\rop,a) \big] \, \tm (\la+i\rop) \right) \wrt \la
	\end{align*} 
for $a\in \BA^+,$ and
	\begin{align*}
	\ka_2^p(a)=\frac{\eta_-(a)(-1)^N}{(\log a)^N}\int_{\fra^*} a^{i\la}\partial_{\la}^N\left((1-\bC)(\la)\, \big[1+\om(\la,a) \big] \, \tm (\la) \right) \wrt \la,\qquad a \in \BA^-.
\end{align*} 
Using assumption 1) from Theorem \ref{t: main i} together with \eqref{f: HCest}, \eqref{f: estIonescu} we reach 
	$$|\ka_2^p(a)|\lesssim \frac{(1-\eta(\ap))}{(\log \ap)^N}.$$ This shows that $\ka_2^p\in L^1(\BA)$ so that $\|\ka_2^p\|_{\Cvp{\BA}}<\infty.$

	It remains to estimate $\|\ka_1^p\|_{\Cvp{\BA}}.$ We use \eqref{e ome'} and expand $\ka_1^p=\sum_{\ell=0}^{\infty} \ka_{1,\ell}^p$ where
	\begin{align*}
		\ka_{1,\ell}^p(a)=&
			\eta_{+,\ell}(a)\int_{\fra^*} a^{i\la}\Gamma_{\ell}(\la+i\rop)  \, \bC(\la) \tm (\la+i\rop)  \wrt \la\\
			&+
			\eta_{-,\ell}(a)\int_{\fra^*} a^{-i\la}\Gamma_{\ell}(\la) \, \bC(\la) \tm  (\la)\wrt \la,\qquad a\in \BA.
	\end{align*}
with $\eta_{\pm,\ell}$ given by \eqref{e epm}. Thus, using Lemma \ref{lem Mel}, \eqref{e epmles}, and  \eqref{e GjMpnorm} we obtain
\begin{equation}
	\label{e k1lpcvp}
	\|\ka_{1,\ell}^p\|_{\Cvp{\BA}}\lesssim \ell^d e^{-\ell}(\|\bC \tm\|_{\Mp{\BA}}+\|\bC \tm(\cdot +i\rop)\|_{\Mp{\BA}}).
\end{equation}
Since $\bC$ is a compactly supported smooth function so are $\bC(\la) \check\bc^{-1}(\la)$ and $\bC(\la) \check\bc^{-1}(\la+i\rop).$  Recalling \eqref{e tmbc} we write $\bC(\la)\tm(\la)=[\bC(\la)\check\bc^{-1}(\la)]m_B(\la)$ and $$ \bC(\la)\tm(\la+i\rop)=[\bC(\la)\check\bc^{-1}(\la+i\rop)]m_B(\la+i\rop)$$ and estimate
\[
\|\bC \tm\|_{\Mp{\BA}}+\|\bC \tm(\cdot +i\rop)\|_{\Mp{\BA}}\lesssim \|m_B\|_{\Mp{\BA}}+\|m_B(\cdot +i\rop)\|_{\Mp{\BA}}.
\]
Here we used the observation that multiplication by a smooth compactly supported function preserves the class of $L^p(\R)$ multipliers, see \cite[Theorem 2.8]{MW} ii) and \cite{Co}. 
Now, coming back to \eqref{e k1lpcvp} and using Lemma \ref{lem:locm} together with the assumption 1) from Theorem \ref{t: main i} we reach
\[
\|\ka_{1,\ell}^p\|_{\Cvp{\BA}}\lesssim \ell^d e^{-\ell}\|m_B(\cdot +i\rop)\|_{\Mp{\BA}}\lesssim  \ell^d e^{-\ell}.
\]
Summing over $\ell=0,1,2,\ldots,$ we obtain $\|\ka_1^p\|_{\Cvp{\BA}}<\infty,$ as desired.

This completes the proof of  Proposition \ref{pro kap i}.
	\end{proof}

\begin{proof}[Proof of Proposition \ref{pro kaq i}]
	Similarly to the proof of  Proposition \ref{pro kaq}, by the mean value theorem it is enough to show that $D \ka^p$ is a bounded function  on $\BA.$ We estimate separately the terms $D \ka_j^p$ $j=1,2,$ in \eqref{e kapsplit}.

	We start with $\ka^p_1$ and consider first $a\in \BA^+.$ Using Leibniz rule we obtain a formula similar to \eqref{e kaqform} with the addition of $\bC(\la)$ under the integral 
	\begin{equation*}
		\begin{split}
			D\ka_1^p(a)=&2 D(1-\eta)(a)\int_{\fra^*} a^{i\la}\,\bC(\la)\,  \big[1+\om(\la+i\rop,a) \big] \, \tm(\la+i\rop)  \wrt \la\\
			&+ 2(1-\eta)(a) \int_{\fra^*} a^{i\la}\,\bC(\la)\,  i\la\big[1+\om(\la+i\rop,a) \big] \, \tm(\la+i\rop)\wrt \la\\
			&+ 2(1-\eta)(a)\int_{\fra^*} a^{i\la}\,\bC(\la)\, D\big[1+\om(\la+i\rop,a) \big] \, \tm(\la+i\rop)\wrt \la.
		\end{split}
	\end{equation*}
Since $m_B$ is a bounded function on $T_p,$ using \eqref{f: HCest}, \eqref{f: estIonescu} we thus see that
\[
|D\ka_1^p(a)|\lesssim 1.
\]
If $a\in \BA^-$ then we proceed as above using the corresponding definition of $\ka^p_1(a)$.

We now move to the estimate for $D\ka_2^p(a).$ Consider first the case $a\in \BA^+$. Here we integrate by parts $N$ times in $\la$ and then apply Leibniz rule. This gives a formula similar to \eqref{e kaqform} with the additional term $1-\bC(\la)$ under the integral
\begin{equation*}
	\begin{split}&D\ka^p_2(a)\\
		=&2(-i)^N D\left(\frac{(1-\eta(a))}{(\log a)^N}\right)\int_{\fra^*} a^{i\la}\,  \partial_{\la}^N\left((1-\bC(\la))\big[1+\om(\la+i\rop,a) \big] \, \tm(\la+i\rop)\right)   \wrt \la\\
		&+ 2(-i)^N\frac{(1-\eta(a))}{(\log a)^N} \int_{\fra^*} a^{i\la}\, i\la\,\partial_{\la}^N\left((1-\bC(\la))\big[1+\om(\la+i\rop,a) \big] \, \tm(\la+i\rop)\right)\\
		&+ 2(-i)^N\frac{(1-\eta(a))}{(\log a)^N}\int_{\fra^*} a^{i\la}\,D\partial_{\la}^N\left((1-\bC(\la))\big[1+\om(\la+i\rop,a) \big] \, \tm(\la+i\rop)\right).
	\end{split}
\end{equation*}
In the case when $a\in \BA^-$ we obtain a similar formula with $(1-\eta(a))(\log a)^{-N}$ replaced by $(1-\eta(\ap))(\log \ap)^{-N}a^{2\rho/p}$ and $\la +i\rop$ replaced by $\la$ in relevant places under the integral. In both cases, using assumption 2) of Theorem \ref{t: main i} together with \eqref{f: HCest}, \eqref{f: estIonescu} we reach
\[
|D\ka_2^p(a)|\lesssim 1,\qquad a\in \BA.
\]
This completes the proof of Proposition \ref{pro kaq i}.

\end{proof}

\subsection*{Acknowledgments}
The author is most grateful to Stefano Meda for a suggestion of study along the lines of the main theorems,  many discussions on the subject and helpful remarks. The research was supported by National Science Centre, Poland (NCN), research project 2018\slash 31\slash B\slash ST1\slash 00204.


\begin{thebibliography}{ccccccc}

\bibitem{A1} J.-Ph. Anker,
\textit{$L_p$ Fourier multipliers on Riemannian symmetric spaces
of the noncompact type}, Ann. of Math. \textbf{132} (1990),
597--628.

%\bibitem{A2} J.-Ph. Anker,
%\textit{Sharp estimates for some functions of the Laplacian
%on noncompact symmetric spaces}, Duke Math. J. \textbf{65} (1992), 257--297.
%
%\bibitem{AJ} J.-Ph. Anker and L. Ji, \textit{Heat kernel and Green
%function estimates on noncompact symmetric spaces I},
%Geom. Funct. Anal. \textbf{9} (1999), 1035--1091.

\bibitem{AL} J.-Ph. Anker and N. Lohou\'e,
\textit{Multiplicateurs sur certain espaces sym\'etriques},
Amer. J. Math \textbf{108} (1986), 1303--1354.

%\bibitem[CMM]{CMM} A. Carbonaro, G. Mauceri and S. Meda,
%\textit{$H^1$ and $BMO$ for certain nondoubling metric measure spaces},
%to appear in Ann. Sc. Norm. Sup. Pisa, arXiv:0808.0146 [math.FA].

\bibitem{CMW} D. Celotto, S. Meda and B. Wr\'obel, \textit{$L^p$ spherical multipliers on 
homogeneous trees},
Studia Math. {\bf 247} (2019), 175--190.

%\bibitem[CGT]{CGT} J. Cheeger, M. Gromov and M. Taylor,
%\textit{Finite propagation speed, kernel estimates for functions of the
%Laplace operator, and the geometry of complete Riemannian manifolds},
%J. Diff. Geom. \textbf{17} (1982), 15--53.

\bibitem{CS} J.-L. Clerc and E.M. Stein,
\textit{$L^p$ multipliers for noncompact symmetric spaces},
Proc. Nat. Acad. Sci. U. S. A. \textbf{71} (1974), 3911--3912.

%\bibitem[CW]{CW} R.R. Coifman and G. Weiss, \textit{Extensions of Hardy
%spaces and their use in analysis}, Bull. Amer. Math. Soc. \textbf{83} (1977), 569--645.

\bibitem{CW} R.R. Coifman and G. Weiss, \textit{Transference methods in analysis},
Conference Board of the Mathematical Sciences Regional Conference Series in Mathematics,
No. \textbf{31}, American Mathematical Society, Providence, R.I., 1976.

\bibitem{Co} M.G. Cowling, 
\textit{Some applications of Grothendieck's theory of topological tensor products in 
Harmonic Analysis},
Math. Ann. \textbf{232} (1978), 273--285.

%\bibitem{CGM1} M.G. Cowling, S. Giulini and S. Meda,
%\textit{Estimates for functions of the Laplace--Beltrami operator on
%noncompact symmetric spaces. II},
%J. Lie Th. \textbf{5} (1995), 1--14.

%\bibitem[Ga]{Ga} R. Gangolli, \textit{On the Plancherel
%formula and the Paley--Wiener theorem for spherical
%functions on semisimple Lie groups},
%Ann. of Math. \textbf{93} (1971), 150--165.

\bibitem{GV} R. Gangolli and V.S. Varadarajan,
Harmonic Analysis of Spherical Functions on
Real Reductive Groups, Springer-Verlag, 1988.

\bibitem{GMM} S. Giulini, G. Mauceri and S. Meda,
$L^p$ multipliers on noncompact symmetric spaces, \emph{J.
reine angew. Math.} \textbf{482} (1997), 151--175.

%\bibitem[GJT]{GJT} Y. Guivarc'h, L.~Ji and J.C. Taylor,
%Compactifications of symmetric spaces,
%Birkh\"auser 1998.

%\bibitem[HC]{HC} Harish-Chandra,
%\textit{Spherical functions on a semisimple Lie group, I.},
%Amer. J. Math. \textbf{8} (1954), 241--310.

\bibitem{H1} S. Helgason,
Groups and Geometric Analysis,
Academic Press, New York, 1984.

\bibitem{H2} S. Helgason,
Differential Geometry, Lie groups, and Symmetric Spaces,
Academic Press, New York, 1978.

\bibitem{H3} S.~Helgason,
Geometric analysis on symmetric spaces,
Math. Surveys \& Monographs \textbf{39}, Amer. Math. Soc., 1994.

%\bibitem{HR} E. Hewitt and K.A. Ross,
%Abstract Harmonic Analysis , A Series of
%Comprehensive Studies in Mathematics 1 (1979), n. 115, Springer- Verlag.

\bibitem{Ho} L. H\"ormander,
\textit{Estimates for translation invariant operators in $L^p$ spaces},
Acta Math. \textbf{104} (1960), 93--140.

\bibitem{I1} A.D. Ionescu, \textit{Fourier integral operators on noncompact
symmetric spaces of real rank one}, J. Funct. Anal.
\textbf{174} (2000), 274--300.

\bibitem{I2} A.D. Ionescu, \textit{Singular integrals on symmetric spaces
of real rank one}, Duke Math. J.
\textbf{114} (2002), 101--122.

%\bibitem{I3} A.D. Ionescu, \textit{Singular integrals
%on symmetric spaces, II}, Trans. Amer. Math. Soc.
%\textbf{335} (2003), 3359--3378.

%\bibitem{L} N.N. Lebedev, Special Functions and Their Applications, Dover, New York, 1972.

\bibitem{MV} S. Meda and M. Vallarino, \textit{Weak type estimates for spherical
multipliers on noncompact symmetric spaces},
\emph{Trans. Amer. Math. Soc.} (6) \textbf{362} (2010), 2993--3026.

\bibitem{MW} S. Meda and B. Wr\'obel, \textit{Marcinkiewicz-type multipliers on noncompact symmetric spaces}, \emph{Annali della Scuola normale superiore di Pisa - Classe di scienze}, (5) \textbf{22} (2021), 1747--1804. 

\bibitem{Mikhlin} S.\ G.\ Mikhlin, \textit{Multidimensional singular integrals and integral equations}, Translated from
the Russian by W.\ J.\ A.\ Whyte. Pergamon Press, Oxford-New York-Paris 1965.

\bibitem{ST} R.J.~Stanton, P.A.~Tomas,
\textit{Expansions for spherical functions on noncompact symmetric spaces},
Acta Math. \textbf{140} (1978), 251--276.

%\bibitem{St1} E.M. Stein, \textit{Harmonic Analysis. Real variable
%methods, orthogonality and oscillatory integrals}, Princeton Math. Series
%No. \textbf{43}, Princeton N. J., 1993.

\bibitem{SW} E.M. Stein and G. Weiss, \textit{Introduction to Fourier Analysis on Euclidean Spaces}, Princeton University Press, Princeton N.J.\ 1971.

%\bibitem{Str} J.-O.~Str\"omberg,
%\textit{Weak type $L^1$ estimates for maximal functions on noncompact
%symmetric spaces}, Ann. of Math.
%\textbf{114} (1981), 115--126.

%\bibitem[T]{T} M.E. Taylor,
%\textit{$L^p$ estimates on functions of the Laplace operator},
%Duke Math. J. \textbf{58} (1989), 773--793.

%\bibitem{TV} P.C. Trombi and V.S. Varadarajan,
%\textit{Spherical transforms on semisimple Lie groups},
%Ann.\break of Math. \textbf{94} (1971), 246--303.

%\bibitem{W} G.N. Watson, A treatise on the theory of Bessel functions, Cambridge Univ. Press
%Cambridge, 2nd edition, 1944.

%\bibitem{WroJSMMSO} B.\ Wr\'obel,  \textit{ Joint spectral multipliers for mixed systems of operators}, 
%J. Fourier Anal. Appl. (2) {\bf 23} (2017),  245--287.
\end{thebibliography}
\end{document}